\documentclass[12pt]{article} \textwidth=16.2cm \textheight=21cm
\usepackage{amsthm}
\usepackage{color}
\usepackage{cite}
\usepackage{graphicx}
\usepackage{psfrag}
\usepackage{url}
\usepackage{stfloats}
\usepackage{amsmath}
\usepackage{amssymb}
\usepackage{graphicx}
\usepackage[titletoc]{appendix}

\newtheorem{theorem}{Theorem}[section]
\newtheorem{definition}{Definition}[section]
\newtheorem{lemma}{Lemma}[section]
\newtheorem{example}{Example}[section]
\newtheorem{remark}{Remark}[section] 
\newtheorem{proposition}{Proposition}[section]

\def\a{\alpha} \def\b{\beta}

\def\dref#1{(\ref{#1})}

\def\sgn{\mbox{sgn}\,}

 \def\dfrac{\displaystyle\frac}

\def\be{\begin{equation}}
\def\bel{\begin{equation}\label}
\def\ee{\end{equation}}
\def\ba{\begin{array}}
\def\ea{\end{array}}
\def\banl{\begin{eqnarray}\label}
\def\ean{\end{eqnarray}}

 \def\bna{\begin{eqnarray}}
\def\ena{\end{eqnarray}} \def\dref#1{(\ref{#1})}
\def\qed{\strut\hfill $\Box$}

\begin{document}

\title{  Inverse Eigenvalue Problem For Mass-Spring-Inerter Systems  \footnotemark[1]
\author{Zhaobo Liu\footnotemark[2]
, \,Qida Xie\footnotemark[2]
, \,Chanying Li \footnotemark[2]
          }
}

\footnotetext[1]{The work was supported by  the National Natural Science Foundation of China under Grants 11925109 and 11688101.  }

\footnotetext[2]{%
Z. B.~Liu,~Q. D. Xie and C.~Li are with the Key Laboratory of Systems and Control, Academy of Mathematics and Systems Science, Chinese Academy of Sciences, Beijing 100190, P.~R.~China.   They are also with the School of Mathematical Sciences, University of Chinese Academy of Sciences, Beijing 100049, P. R. China.  Corresponding author: Chanying~Li (Email: \texttt{cyli@amss.ac.cn}).}

 \maketitle

\begin{abstract}
This paper has solved the inverse eigenvalue problem for ``fixed-free'' mass-chain systems with inerters. It is well known that for a spring-mass system wherein the adjacent masses are linked  through a spring, the  natural frequency assignment can be achieved by choosing appropriate masses and spring stiffnesses if and only if the given  positive eigenvalues   are distinct. However, when we involve inerters, multiple eigenvalues  in the assignment are allowed. In fact, arbitrarily given a set of  positive real numbers,   we derive a necessary and sufficient condition on the multiplicities of these numbers, which are assigned as the natural frequencies of the concerned mass-spring-inerter system.

\end{abstract}

\section{Introduction}

Natural frequency,  an inherent attribute of mechanical vibration systems, has attracted wide attention for its  importance.
 In particular, purposefully  allocating the natural frequencies to some pre-specified  values provides an effective way to  induce or evade  resonance (see \cite{SPB2006}, \cite{YY2001}). This naturally raises the  inverse eigenvalue problem (IEP),
that is, to construct a  vibration system whose natural frequencies,  or mathematically   known as eigenvalues,    are given beforehand.

A well-known result on this problem is due to \cite{GMLG2005} and \cite{PN1997}, which is  addressed for
  mass-spring systems.  Observe that in such a basic system, the adjacent masses are linked merely by a spring. Therefore, the IEP  turns out to be the construction of a Jacobi matrix with its eigenvalues being assigned to a set of specified positive numbers. Borrowing the tools for Jacobi matrices,
   \cite{GMLG2005} and \cite{PN1997} assert that the IEP  is solvable if and only if the given  positive eigenvalues   are distinct. Later on, a various of inverse problems  on  Jacobi matrices and Jacobi operators are investigated as well. But when  dampers are taken into consideration, the matrices associated with masses, spring stiffnesses and damping coefficients  in  mass-spring-damper systems  are no longer  Jacobi matrices and the quadratic inverse eigenvalue problem (QIEP) is  put forth.  Nevertheless, most of the       literatures on   mass-spring-damper systems  focus on distinct eigenvalues assignment \cite{CAI2009,LIN2010,BAI2007,PN1999}.

   Interestingly, the IEP   admits  multiple   eigenvalues,  if we introduce  a mechanical element called the inerter.  This new mechanical device can  simulate masses by changing inertance. It
    was theoretically first studied by \cite{MCS2002}, completing  the   analogy between electrical and mechanical networks (see  Figure \ref{analogy}). Through physical  realization,  inerters have been applied to many engineering fields such as vibration isolators,   landing gears, train suspensions, building vibration control, and so on \cite{XD2015,FCW2010-T,AZM2015,FCW2010-B,IFL2014}. In a mass-spring-inerter system, the neighbouring masses are linked by a parallel combination of a spring and an inerter. As a starting point, we restrict our interest in this paper to ``fixed-free'' systems. The term  ``fixed-free'' means one end of the mass-chain system     is attached to the ground while  the other end is hanging free, as shown in Figure \ref{Fig.1}.

The free vibration equation of such  a  mass-spring-inerter system  is described by
$$\boldsymbol{(M+B) \ddot{x}}+ \boldsymbol{Kx}=0,$$
where $\boldsymbol{x}=(x_1,x_2,\ldots,x_n)^\mathrm{T}\in \mathbb{R}^n$ and
\begin{eqnarray}\label{MM}
\boldsymbol{M} = \mbox{diag}\lbrace m_1,m_2,\ldots,m_n\rbrace,
\end{eqnarray}
\begin{eqnarray}\label{KK}
\boldsymbol{K}&=&
\begin{bmatrix}
k_1+k_2&     -k_2&      &                            \\
 -k_2   &        k_2+k_3&     -k_3&          \\
                    &  \ddots&    \ddots   &   \ddots  &             \\
                    &            &   -k_{n-1}  & k_{n-1}+k_n    &-k_n   \\
                    &          &     &    -k_n &      k_n
\end{bmatrix},
\end{eqnarray}
\begin{eqnarray}\label{BB}
\boldsymbol{B}&=&
\begin{bmatrix}
b_1+b_2&     -b_2&      &                            \\
 -b_2   &        b_2+b_3&     -b_3&          \\
                    & \ddots &    \ddots   &   \ddots  &             \\
                    &            &    -b_{n-1}  &     b_{n-1}+b_n     &-b_n   \\
                    &          &     &    -b_n &      b_n
\end{bmatrix}.
\end{eqnarray}
  Here,  real numbers $m_j> 0$, $k_j> 0$, $b_j\geq 0$ for $j=1,\ldots,n$ stand for the masses, spring stiffnesses and   inertances. Unlike mass-spring systems,  
the well-studied Jacobi matrix theory cannot illuminate the IEP for mass-spring-inerter systems since the inertial matrix in \dref{BB} is a tridiagonal matrix.
Recently, \cite{chen2018} found  that inerters render the multiple eigenvalues possible for a mass-chain system. It showed that  the  multiplicity $t_i$  of  a natural frequency  $\lambda_i$ must fulfill $n\geq 2t_i-1$. Beyond that, little is known for the multiple eigenvalue  case.

The purpose of this paper is to solve the IEP for mass-spring-inerter systems, where the  eigenvalues are  arbitrarily  specified to $n$ positive real numbers.    We deduce a necessary and sufficient condition for this assignment   on the multiplicities of the given numbers. With the proposed critical criterion, the set structure of the given real numbers will be intuitively clear for the natural frequency assignment. 
Our construction further implies that  $m$ masses  of the system can be arbitrarily  fixed  beforehand for the assignment, where $m$ is the amount of the  distinct assigned  eigenvalues. More precisely, our construction is carried out by only adjusting $n-m$ massess, $n$ spring stiffnesses and $n$ inertances.
It  degenerates to
the claim that, if the pre-specified eigenvalues are all distinct ($m=n$), the IEP can be worked out by recovering $\boldsymbol{K}$ and $\boldsymbol{B}$,  whereas $\boldsymbol{M}$ is  fixed arbitrarily. This claim is exactly the main result of \cite{MCS2002}, which demonstrates an advantage of using inerters in the  mass-fixed situation. Unfortunately, not all the  natural frequency assignments are realizable  by  merely adjusting  spring stiffnesses and inertances. An example of five-degree-of-freedom system in this paper shows that there exist some restrictive relationships between masses and  given eigenvalues.



The organization of this paper is as follows. In Section \ref{MR}, we state the main result by deducing a necessary and sufficient condition of the IEP for mass-spring-inerter systems, while the proofs are included in Sections \ref{n1} and \ref{MR1}. Conclusions are drawn in Section \ref{conre}.



\begin{figure}[htbp]

\begin{minipage}[t]{0.5\linewidth}
\centering
\includegraphics[height=3.0cm,width=8.0cm]{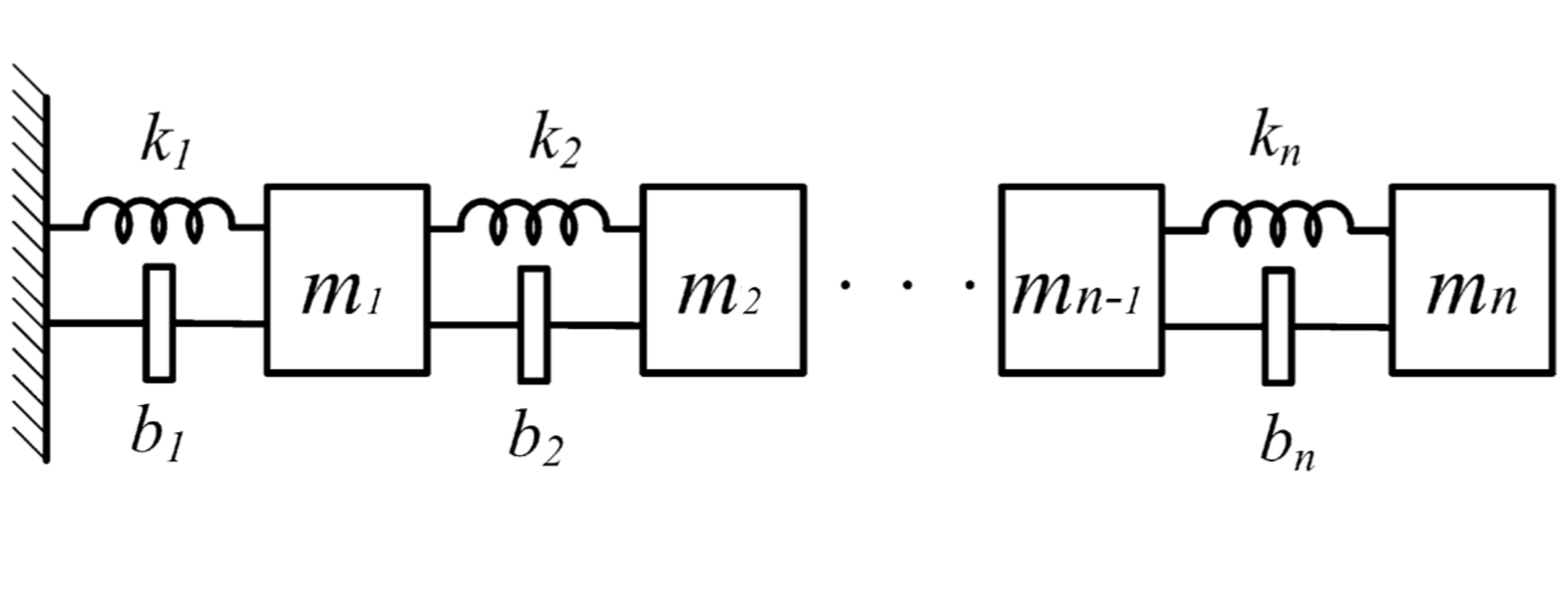}
\caption{Mass-spring-inerter system}
\label{Fig.1}
\end{minipage}
\hfill
\begin{minipage}[t]{0.4\linewidth}
\centering
\includegraphics[height=4.5cm,width=6.3cm]{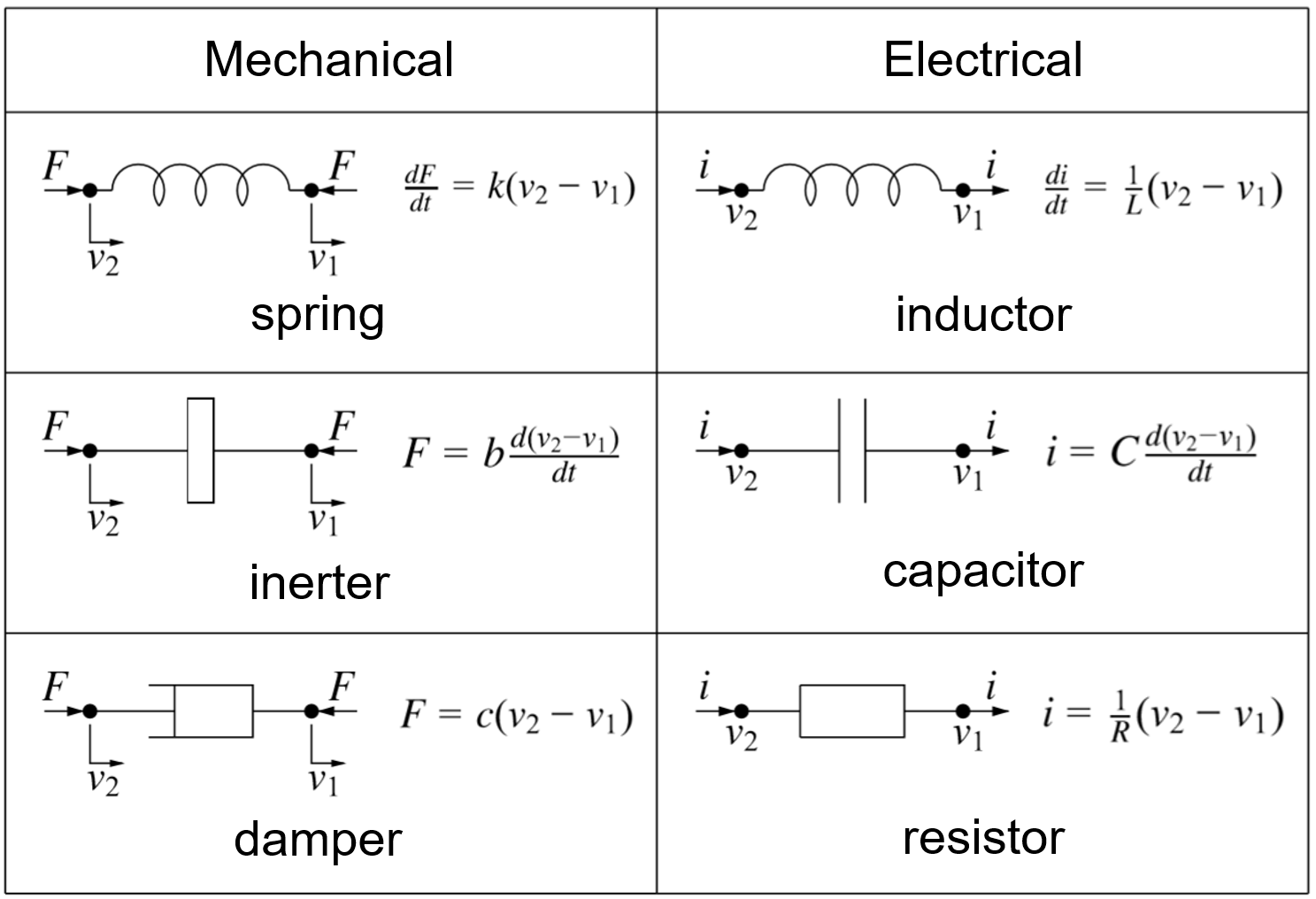}
\caption{The force-current analogy}
\label{analogy}
\end{minipage}

\end{figure}


\section{Main Result}\label{MR}

The natural frequencies of a mass-spring-inerter system are completely determined by
the eigenvalues of  matrix pencil $\boldsymbol{K}-\lambda (\boldsymbol{M}+\boldsymbol{B})$,  where $\boldsymbol{M}, \boldsymbol{K}, \boldsymbol{B}$ are defined by (\ref{MM}), (\ref{KK}) and (\ref{BB}), respectively. So, with a slight abuse of  language,  we  will not distinguish the term ``eigenvalues" from the ``natural frequencies"  in this article.
We now raise our problem.

\textbf{Problem 1.} Arbitrarily given a set of real numbers $0<\lambda_1 \leq \lambda_2 \leq \cdots \leq \lambda_{n}$, is it possible to recover matrices $\boldsymbol{M}, \boldsymbol{K}, \boldsymbol{B}$ in (\ref{MM}), (\ref{KK}) and (\ref{BB}) by
 choosing  $m_j> 0$, $k_j> 0$ and $b_j \geq 0$  for  $j=1,\ldots, n$,  so that the $n$ eigenvalues of  matrix pencil $\boldsymbol{K}-\lambda (\boldsymbol{M}+\boldsymbol{B})$ are exactly $\lambda_i$, $i=1,\ldots, n$?

Both \cite{chen2018} and \cite{PN1997} offered a positive answer to Problem 1 for the special case where the eigenvalues are all  distinct.
But the  general situation should involve multiple eigenvalues, which is covered by the following theorem.




\begin{theorem}\label{sta3}
Let $\prod_{i=1}^{m}(\lambda -\lambda_i)^{t_i}$ be a  polynomial with $0<\lambda_1<\lambda_2<\cdots<\lambda_{m}$ and $\sum_{i=1}^{m}{t_i}=n$. Then,
there exist some matrices $\boldsymbol{K},\boldsymbol{M}, \boldsymbol{B}$ in the forms of \dref{MM}--\dref{BB} such that
\begin{eqnarray}\label{KMBroots}
\prod_{i=1}^{m}(\lambda -\lambda_i)^{t_i} \Big| \det(\boldsymbol{K}-\lambda (\boldsymbol{M}+\boldsymbol{B}))
\end{eqnarray}
\textbf{if and only if}
\begin{eqnarray}\label{jbtj}
t_i\leq i,\quad i=1,\ldots,m.
\end{eqnarray}

\end{theorem}

\begin{remark}
Theorem \ref{sta3} completely solves Problem 1  by providing the critical criterion \dref{jbtj}. As indicated later (see Proposition \ref{gd} for details), when \dref{jbtj} holds,   the   recover of the relevant matrices  allows a total of $m$ masses being taken arbitrarily, where $m$ is the number of distinct eigenvalues $\lambda_i$ given beforehand. Particularly, for $m=n$, each mass $m_i, 1\leq i\leq n$ can be taken any  fixed quantity in advance, as proved in \cite{chen2018}. However,
when $m<n$,  it is generally impossible to   achieve the natural frequency assignment  with all  the   masses   arbitrarily given. Example \ref{thm3} suggests a restrictive relation between the masses and eigenvalues.


\end{remark}

\begin{example}\label{thm3}
Let $n=5$, $t_1=t_2=1$, $t_3=3$ and  $0<\lambda_1<\lambda_2<\lambda_3$.  If there exist some $m_j>0$, $k_j>0$, $b_j\geq 0$, $j=1, 2,\ldots,5$ such that $\prod_{i=1}^{3}(\lambda -\lambda_i)^{t_i} \Big| \det(\boldsymbol{K}-\lambda (\boldsymbol{M}+\boldsymbol{B}))$, then
\begin{eqnarray}\label{exp1m}
\max_{j\in[2,4]}\frac{m_j}{m_{j+1}}>\frac{\lambda_1}{8\lambda_3}\left(1-\left(\frac{\lambda_2}{\lambda_3}\right)^{\frac{1}{3}}\right).
\end{eqnarray}
Clearly, \dref{jbtj} holds, but the masses cannot be taken arbitrarily.
The proof of \dref{exp1m} is provided in Appendix \ref{appa}.
\end{example}
\section{Proof of the necessity of Theorem \ref{sta3}.}\label{n1}
This section is devoted to proving the necessity of Theorem \ref{sta3}, which  is relatively easier  than the  argument for sufficiency.
We begin by expressing $\det(\boldsymbol{K}-\lambda (\boldsymbol{M}+\boldsymbol{B}))$ in terms of a recursive sequence of  polynomials. First,  write
\begin{eqnarray*}
&&\boldsymbol{K}-\lambda (\boldsymbol{M}+\boldsymbol{B})=\nonumber\\
&&\begin{bmatrix}
 \begin{smallmatrix}
k_1+k_2-\lambda (m_1+b_1+b_2)&     -k_2+\lambda b_2&                                    \\
 -k_2+\lambda b_2   &        k_2+k_3-\lambda (m_2+b_2+b_3)&     -k_3+\lambda b_3          \\
                    & \ddots &    \ddots   &        \ddots            \\
                    &            &   -k_{n-1}+\lambda b_{n-1} &   k_{n-1}+k_n-\lambda(m_{n-1}+b_{n-1}+b_n)   &-k_n+\lambda b_n   \\
                    &          &     &    -k_n+\lambda b_n &      k_n-\lambda (m_n+b_n) \\
  \end{smallmatrix}
\end{bmatrix}.\nonumber
\end{eqnarray*}
For $j=1,\ldots,n$,  let $\boldsymbol{M_j}$, $\boldsymbol{K_j}$ and $\boldsymbol{B_j}$ be some matrices  defined analogously as $\boldsymbol{M},\boldsymbol{K}$ and $\boldsymbol{B}$ in (\ref{MM})--(\ref{BB}), respectively,  but with order $j$ instead of $n$.
Next, denote $f_j(\lambda )$  as the determinant of $\boldsymbol{K_j}-\lambda (\boldsymbol{M_j}+\boldsymbol{B_j})$, $j=1,\ldots,n$. Let $g_1(\lambda)=1$ and
 $g_{j}(\lambda)$  the  leading principal minor of  $\boldsymbol{K_j}-\lambda (\boldsymbol{M_j}+\boldsymbol{B_j})$    of  order $j-1$, where $j=2,\ldots,n$.
 So, to calculate $\det(\boldsymbol{K}-\lambda (\boldsymbol{M}+\boldsymbol{B}))$,
 we only need to treat $f_n(\lambda)$.
\begin{remark}
For each $j=1,\ldots,n$,
since $\det\boldsymbol{K_j}= \prod_{l=1}^{j}k_l> 0$,  the Gershgorin's circle theorem indicates that
 both $\boldsymbol{K_j}$ and $\boldsymbol{M_j}+\boldsymbol{B_j}$ are positive definite matrices and   so do their leading principal submatrices. Then,  it follows that the roots of $f_j(\lambda )$ and $g_j(\lambda )$ are all real and positive (see \cite[Theorem 1.4.3]{GMLG2005} ).
\end{remark}
To facilitate the subsequent analysis, we introduce  the following  definition.
\begin{definition}\label{f<<g}
Let $f(\lambda)$ and $g(\lambda)$ be two polynomials with degree $s$, where $s\in \mathbb{N}^+$. Suppose $f(\lambda)$ and $g(\lambda)$ both have $s$ distinct real roots, which are denoted  by
 $\alpha_1<\cdots<\alpha_s$ and $\beta_1<\cdots<\beta_s$, respectively.  We say $g(\lambda)\ll  f(\lambda)$,  if  their  leading coefficients are of the same sign and
$$\beta_1<\alpha_1<\beta_2<\alpha_2<\cdots<\beta_s<\alpha_s.$$
\end{definition}

The proof depends on a simple observation below.

\begin{lemma}\label{ditui}
The polynomials $\lbrace f_{j}(\lambda)\rbrace_{j=1}^{n}$ and $\lbrace g_{j}(\lambda)\rbrace_{j=1}^{n}$  satisfy
\begin{eqnarray}\label{LS}
\left\{
\begin{aligned}
&f_{j+1}(\lambda )=(-\lambda m_{j+1})g_{j+1}(\lambda)+(k_{j+1}-\lambda b_{j+1})f_{j}(\lambda )\\
&g_{j+1}(\lambda)=f_{j}(\lambda )+(k_{j+1}-\lambda b_{j+1})g_{j}(\lambda)
\end{aligned},\quad j=1,\ldots,n-1
\right.
\end{eqnarray}
with $g_1(\lambda)=1$ and $ f_1(\lambda)=k_1-\lambda (m_1+b_1)$.
\end{lemma}

\begin{proof}
By the definition of $\boldsymbol{K_{1}}-\lambda (\boldsymbol{M_{1}}+\boldsymbol{B_{1}})$, it is trivial that
$ f_1(\lambda)=k_1-\lambda (m_1+b_1)$. For $j=1,\ldots, n-1$, expanding the leading principal minor of  $\boldsymbol{K_{j+1}}-\lambda (\boldsymbol{M_{j+1}}+\boldsymbol{B_{j+1}})$ of order $j$ by cofactors of the $j$th row shows  $g_{j+1}(\lambda)=f_{j}(\lambda )+(k_{j+1}-\lambda b_{j+1})g_{j}(\lambda)$. Furthermore, the expansion of $\det (\boldsymbol{K_{j+1}}-\lambda (\boldsymbol{M_{j+1}}+\boldsymbol{B_{j+1}}))$ by cofactors of the $(j+1)$th row yields
\begin{eqnarray*}
f_{j+1}(\lambda )
&=&(k_{j+1}-\lambda (m_{j+1}+b_{j+1}))g_{j+1}(\lambda)-(k_{j+1}-\lambda b_{j+1})^2g_{j}(\lambda )\\
&=&(-\lambda m_{j+1})g_{j+1}(\lambda)+(k_{j+1}-\lambda b_{j+1})f_{j}(\lambda )\\
&&+(k_{j+1}-\lambda b_{j+1})^2g_{j}(\lambda )-(k_{j+1}-\lambda b_{j+1})^2g_{j}(\lambda )\\
&=&(-\lambda m_{j+1})g_{j+1}(\lambda)+(k_{j+1}-\lambda b_{j+1})f_{j}(\lambda ),
\end{eqnarray*}
as desired.
\end{proof}

\begin{lemma}\label{one}
Let $f(\lambda)$ and $g(\lambda)$ be two polynomials 
that $g(\lambda)\ll f(\lambda)$, then for any  $a,b>0$, $g(\lambda)\ll ag(\lambda)+bf(\lambda)\ll f(\lambda)$.
\end{lemma}

\begin{proof}
Since $g(\lambda)\ll f(\lambda)$, considering Definition \ref{f<<g}, we let $\deg(f(\lambda))=\deg (g(\lambda))=s$ for some $s\in \mathbb{N}^+$ and let $\{\a_i\}_{i=1}^s$ and $\{\b_i\}_{i=1}^s$ be the roots of $f(\lambda)$ and $g(\lambda)$, respectively. Clearly,
$$\beta_1<\alpha_1<\beta_2<\alpha_2<\cdots<\beta_s<\alpha_s.$$
 Without loss of generality, assume  the leading coefficients of $f(\lambda)$ and $g(\lambda)$ are both positive. Then,
  $$\sgn((ag(\lambda)+bf(\lambda))(\beta_i))=(-1)^{s+1-i}=-(-1)^{s-i}=-\sgn((ag(\lambda)+bf(\lambda))(\alpha_i)).$$
  This implies that for each $1\leq i\leq s$, there is a root of $ag(\lambda)+bf(\lambda)$ falling in interval $(\beta_i,\alpha_i)$. Observing that the degree of $ag(\lambda)+bf(\lambda)$ is $s$, the result follows immediately.
\end{proof}

We present  an important property enjoyed by sequence $\lbrace f_{j}(\lambda), g_{j}(\lambda)\rbrace_{j=1}^{n}$.

\begin{lemma}\label{two}
 Suppose for some $j\in [1,n-1]$,
\begin{eqnarray}\label{tj}
\frac{(-\lambda) g_j(\lambda)}{( f_j(\lambda), g_j(\lambda))}\ll \frac{ f_j(\lambda)}{( f_j(\lambda), g_j(\lambda))},
\end{eqnarray}
then
$\frac{(-\lambda)g_{j+1}(\lambda)}{(f_{j+1}(\lambda),g_{j+1}(\lambda))}\ll \frac{f_{j+1}(\lambda)}{(f_{j+1}(\lambda),g_{j+1}(\lambda))}.$
Moreover, \\
(i) if $b_{j+1}\neq 0$ and $k_{j+1}-\lambda b_{j+1} | \frac{ f_j(\lambda)}{( f_j(\lambda), g_j(\lambda))}$, then $(f_{j+1}(\lambda),g_{j+1}(\lambda))=(f_{j}(\lambda),g_{j}(\lambda))(\lambda-\frac{k_{j+1}}{b_{j+1}})$
and $\frac{f_{j+1}(\lambda)}{(f_{j+1}(\lambda),g_{j+1}(\lambda))}\ll \frac{- f_j(\lambda)}{( f_j(\lambda), g_j(\lambda))}$;\\
(ii) if $b_{j+1}\neq 0$ and $k_{j+1}-\lambda b_{j+1}\nmid \frac{ f_j(\lambda)}{( f_j(\lambda), g_j(\lambda))}$, then
$(f_{j+1}(\lambda),g_{j+1}(\lambda))=(f_{j}(\lambda),g_{j}(\lambda))$ and $\frac{f_{j+1}(\lambda)}{(f_{j+1}(\lambda),g_{j+1}(\lambda))}\ll \frac{ f_j(\lambda)(k_{j+1}-\lambda b_{j+1})}{( f_j(\lambda), g_j(\lambda))}$;\\
(iii) if $b_{j+1}=0$, then $(f_{j+1}(\lambda),g_{j+1}(\lambda))=(f_{j}(\lambda),g_{j}(\lambda))$ and $\frac{(-\lambda)f_{j}(\lambda)}{(f_{j}(\lambda),g_{j}(\lambda))}\ll \frac{(-\lambda)g_{j+1}(\lambda)}{(f_{j+1}(\lambda),g_{j+1}(\lambda))}\ll \frac{f_{j+1}(\lambda)}{(f_{j+1}(\lambda),g_{j+1}(\lambda))}$.
\end{lemma}

\begin{proof}
First, according to the definitions of  $\lbrace f_{j}(\lambda), g_{j}(\lambda)\rbrace_{j=1}^{n}$, it is apparent that
$\deg(f_j(\lambda))=j$  and $\deg(g_j(\lambda))=j-1$  for each $j\in [1,n]$. Let $\deg((f_j(\lambda),g_j(\lambda)))=j-s_j$ for some integer  $ s_j\in [1,j]$. Recalling \dref{tj},  denote the roots of $\frac{f_j(\lambda)}{(f_j(\lambda),g_j(\lambda))}$ and $\frac{(-\lambda)g_j(\lambda)}{(f_j(\lambda),g_j(\lambda))}$ by $\alpha_{j,1}<\cdots<\alpha_{j,s_j}$ and $0<\beta_{j,1}<\cdots<\beta_{j,s_{j}-1}$, respectively. These roots fulfill
\begin{eqnarray}\label{ic}
0<\alpha_{j,1}<\beta_{j,1}<\cdots<\alpha_{j,s_j-1}<\beta_{j,s_{j}-1}<\alpha_{j,s_j}.
\end{eqnarray}
In addition, the second equation of \dref{LS} implies
 $(f_j(\lambda),g_j(\lambda)) | (f_j(\lambda),g_{j+1}(\lambda))$. Now, we
 prove this lemma by discussing three cases.

(i)  $b_{j+1}\neq 0$ and $k_{j+1}-\lambda b_{j+1} | \frac{f_j(\lambda)}{(f_j(\lambda),g_j(\lambda))}$. For this case,
there exists some $1\leq t_j\leq s_j$ such that $\alpha_{j,t_j}=\frac{k_{j+1}}{b_{j+1}}$. We shall evaluate the sign of $\frac{g_{j+1}(\lambda)}{(f_{j}(\lambda),g_{j}(\lambda))(k_{j+1}-\lambda b_{j+1})}$ at $\a_{j,i}, 1\leq i\leq s_j$. In fact, according to the second equation of \dref{LS},
\begin{eqnarray*}
&&\frac{g_{j+1}(\lambda)}{(f_{j}(\lambda),g_{j}(\lambda))(k_{j+1}-\lambda b_{j+1})}(\alpha_{j,i})\nonumber\\
&=&\frac{f_{j}(\lambda)}{(f_{j}(\lambda),g_{j}(\lambda))(k_{j+1}-\lambda b_{j+1})}(\alpha_{j,i})+\frac{g_{j}(\lambda)}{(f_{j}(\lambda),g_{j}(\lambda))}(\alpha_{j,i}),\qquad 1\leq i\leq s_j.
\end{eqnarray*}
Note that the leading coefficient of $(f_j(\lambda),g_j(\lambda))$ is positive, the definitions of $f_j(\lambda)$ and $g_j(\lambda)$ read
\begin{eqnarray}\label{fjgj=}
\left\{
\begin{aligned}
&f_j(\lambda) = (-1)^{j-s_j}(f_j(\lambda),g_j(\lambda))\prod_{l=1}^{s_j}(\alpha_{j,l}-\lambda)\\
&g_j(\lambda) = (-1)^{j-s_j}(f_j(\lambda),g_j(\lambda))\prod_{l=1}^{s_j-1}(\beta_{j,l}-\lambda)
\end{aligned},
\right.
\end{eqnarray}
which, together with
\dref{ic}, yields that for $1\leq i\leq s_j$,
\begin{eqnarray*}
\sgn\left(\frac{f_{j}(\lambda)}{(f_{j}(\lambda),g_{j}(\lambda))(k_{j+1}-\lambda b_{j+1})}(\alpha_{j,i})\right)& = &(-1)^{j-s_{j}}\sgn\left(\prod_{l=1,l\neq t_j}^{s_j}(\alpha_{j,l}-\alpha_{j,i})\right)\\
&=& \left\{
\begin{array}{lcl}
0,& i\neq t_j\\
(-1)^{j-s_{j}+t_j-1},& i = t_j
\end{array}
\right.
\end{eqnarray*}
and
\begin{eqnarray*}
\sgn\left(\frac{g_{j}(\lambda)}{(f_{j}(\lambda),g_{j}(\lambda))}(\alpha_{j,i})\right) = (-1)^{j-s_{j}}\sgn\left(\prod_{l=1}^{s_{j-1}}(\beta_{j,l}-\alpha_{j,i})\right) = (-1)^{j-s_{j}+i-1}.
\end{eqnarray*}
Therefore,
$$\sgn\left(\frac{g_{j+1}(\lambda)}{(f_{j}(\lambda),g_{j}(\lambda))(k_{j+1}-\lambda b_{j+1})}(\alpha_{j,i})\right)= (-1)^{j-s_{j}+i-1} \quad \mbox{for}\quad 1\leq i\leq s_j.$$
This means that for each $i\in [1,s_j-1]$,  there exists exactly one root of $\frac{g_{j+1}(\lambda)}{(f_{j}(\lambda),g_{j}(\lambda))(k_{j+1}-\lambda b_{j+1})}$ between $\a_{j,i}$ and $\a_{j,i+1}$,
 and hence $ \frac{(-\lambda) g_{j+1}(\lambda)}{(f_{j}(\lambda),g_{j}(\lambda))(k_{j+1} -\lambda b_{j+1})}\ll\frac{f_{j}(\lambda)}{(f_{j}(\lambda),g_{j}(\lambda))}$.

 Moreover,  $(f_j(\lambda),g_j(\lambda)) | (f_j(\lambda),g_{j+1}(\lambda))$,  so $$(f_{j}(\lambda),g_{j+1}(\lambda))=(f_{j}(\lambda),g_{j}(\lambda))\left(\lambda-\frac{k_{j+1}}{b_{j+1}}\right)$$
  and
\begin{eqnarray*}
\frac{(-\lambda)g_{j+1}(\lambda)}{(f_{j}(\lambda),g_{j+1}(\lambda))}=\frac{(-\lambda)g_{j+1}(\lambda)}{(f_{j}(\lambda),g_{j}(\lambda))\left(\lambda -\frac{k_{j+1}}{b_{j+1}}\right)}\ll\frac{-f_{j}(\lambda)}{(f_{j}(\lambda),g_{j}(\lambda))}=\frac{\left(\frac{k_{j+1}}{b_{j+1}}-\lambda\right)f_{j}(\lambda)}{(f_{j}(\lambda),g_{j+1}(\lambda))}.
\end{eqnarray*}
Applying Lemma \ref{one} to the first equation of \dref{LS}, we thus deduce
\begin{eqnarray}\label{gj+1fj}
\frac{(-\lambda)g_{j+1}(\lambda)}{(f_{j}(\lambda),g_{j+1}(\lambda))}\ll \frac{f_{j+1}(\lambda)}{(f_{j}(\lambda),g_{j+1}(\lambda))}\ll \frac{\left(\frac{k_{j+1}}{b_{j+1}}-\lambda\right)f_{j}(\lambda)}{(f_{j}(\lambda),g_{j+1}(\lambda))} = \frac{-f_j(\lambda)}{(f_j(\lambda),g_j(\lambda))}.
\end{eqnarray}
Now,  $(f_{j+1}(\lambda),g_{j+1}(\lambda))=(f_{j}(\lambda),g_{j+1}(\lambda))$ 
 becasue of  $\frac{(-\lambda)g_{j+1}(\lambda)}{(f_{j}(\lambda),g_{j+1}(\lambda))}\ll \frac{f_{j+1}(\lambda)}{(f_{j}(\lambda),g_{j+1}(\lambda))}$,  it follows from \dref{gj+1fj} that
$$\frac{(-\lambda)g_{j+1}(\lambda)}{(f_{j+1}(\lambda),g_{j+1}(\lambda))}\ll \frac{f_{j+1}(\lambda)}{(f_{j+1}(\lambda),g_{j+1}(\lambda))}\ll \frac{-f_j(\lambda)}{(f_j(\lambda),g_j(\lambda))}.$$

(ii) $b_{j+1}\neq 0$ and $(k_{j+1}-\lambda b_{j+1})\nmid \frac{f_j(\lambda)}{(f_j(\lambda),g_j(\lambda))}$.  We first assume $\frac{k_{j+1}}{b_{j+1}}\in (\alpha_{j,t_j},\beta_{j,t_j})$ for some $1\leq t_j\leq s_j-1$ and evaluate the sign of $\frac{g_{j+1}(\lambda)}{(f_{j}(\lambda),g_{j}(\lambda))}$ at points  $\frac{k_{j+1}}{b_{j+1}}$ and  $\a_{j,i}, 1\leq i\leq s_j$. Since Lemma \ref{ditui} shows
$$\frac{g_{j+1}(\lambda)}{(f_{j}(\lambda),g_{j}(\lambda))}(\alpha_{j,i})=(k_{j+1}-\alpha_{j,i}b_{j+1})\frac{g_{j}(\lambda)}{(f_{j}(\lambda),g_{j}(\lambda))}(\alpha_{j,i}),\quad 1\leq i\leq s_j,$$
by \dref{ic} and \dref{fjgj=},
\begin{eqnarray*}
\sgn\left(\frac{g_{j+1}(\lambda)}{(f_{j}(\lambda),g_{j}(\lambda))}(\alpha_{j,i})\right)&=& (-1)^{j-s_{j}}\sgn\left((k_{j+1}-\alpha_{j,i}b_{j+1})\prod_{l=1}^{s_{j-1}}(\beta_{j,l}-\alpha_{j,i})\right)\\&=&
\left\{
\begin{array}{lcl}
(-1)^{j-s_{j}+i-1},& 1\leq i\leq t_j \\
(-1)^{j-s_{j}+i},&  t_j+1\leq i\leq s_j
\end{array}.
\right.
\end{eqnarray*}
Similarly, by Lemma \ref{ditui}, \dref{ic} and \dref{fjgj=},
\begin{eqnarray*}
&&\sgn\left(\frac{g_{j+1}(\lambda)}{(f_{j}(\lambda),g_{j}(\lambda))}\left(\frac{k_{j+1}}{b_{j+1}}\right)\right)\\
&=& \sgn\left(\frac{f_{j}(\lambda)}{(f_{j}(\lambda),g_{j}(\lambda))}\left(\frac{k_{j+1}}{b_{j+1}}\right)+\frac{(k_{j+1}-\lambda b_{j+1})g_{j}(\lambda)}{(f_{j}(\lambda),g_{j}(\lambda))}\left(\frac{k_{j+1}}{b_{j+1}}\right)\right)= (-1)^{j-s_{j}+t_j}.
\end{eqnarray*}
So, for each $i\in [1,s_j-1]$ with $i\neq t_j$,  there is a root of $\frac{g_{j+1}(\lambda)}{(f_{j}(\lambda),g_{j}(\lambda))}$ falling in $(\alpha_{j,i},\alpha_{j,i+1})$, and the rest two roots of $\frac{g_{j+1}(\lambda)}{(f_{j}(\lambda),g_{j}(\lambda))}$ lie in $(\alpha_{j,t_j},\frac{k_{j+1}}{b_{j+1}})$ and $(\frac{k_{j+1}}{b_{j+1}},\alpha_{j,t_j+1})$, respectively. This infers $(f_{j}(\lambda),g_{j+1}(\lambda))=(f_{j}(\lambda),g_{j}(\lambda))$ by noting that
 $(f_j(\lambda),g_j(\lambda)) | (f_j(\lambda),g_{j+1}(\lambda))$. Hence,
$\frac{(-\lambda)g_{j+1}(\lambda)}{(f_{j}(\lambda),g_{j+1}(\lambda))}\ll \frac{f_j(\lambda)(k_{j+1}-\lambda b_{j+1})}{(f_j(\lambda),g_{j+1}(\lambda))}$ and
applying Lemma \ref{one} to the first equation of \dref{LS}, one has
$$\frac{(-\lambda)g_{j+1}(\lambda)}{(f_{j}(\lambda),g_{j+1}(\lambda))}\ll \frac{f_{j+1}(\lambda)}{(f_{j}(\lambda),g_{j+1}(\lambda))}\ll \frac{f_j(\lambda)(k_{j+1}-\lambda b_{j+1})}{(f_j(\lambda),g_{j+1}(\lambda))}.$$
Then, $(f_{j+1}(\lambda),g_{j+1}(\lambda))=(f_{j}(\lambda),g_{j+1}(\lambda))=(f_{j}(\lambda),g_{j}(\lambda))$ and consequently
\begin{eqnarray*}
\frac{(-\lambda)g_{j+1}(\lambda)}{(f_{j+1}(\lambda),g_{j+1}(\lambda))}&=&\frac{(-\lambda)g_{j+1}(\lambda)}{(f_{j}(\lambda),g_{j+1}(\lambda))}\\
&\ll &\frac{f_{j+1}(\lambda)}{(f_{j+1}(\lambda),g_{j+1}(\lambda))}\ll \frac{f_j(\lambda)(k_{j+1}-\lambda b_{j+1})}{(f_j(\lambda),g_{j+1}(\lambda))}= \frac{f_j(\lambda)(k_{j+1}-\lambda b_{j+1})}{(f_j(\lambda),g_j(\lambda))}.
\end{eqnarray*}

As for the situations where  $\frac{k_{j+1}}{b_{j+1}}\in(0,\alpha_{j,1})$, $\frac{k_{j+1}}{b_{j+1}}\in\bigcup_{l=1}^{s_j-1}[\beta_{j,l},\alpha_{j,l+1})$ and $\frac{k_{j+1}}{b_{j+1}}\in(\alpha_{j,s_j},+ \infty)$, an analogous  treatment can be employed.

(iii)  $b_{j+1}=0$.  We also first  calculate the sign of $\frac{g_{j+1}(\lambda)}{(f_{j}(\lambda),g_{j}(\lambda))}$ at the roots of $\frac{f_j(\lambda)}{(f_j(\lambda),g_j(\lambda)) }$ and $\frac{g_j(\lambda)}{(f_j(\lambda),g_j(\lambda)) }$. As before,
\begin{eqnarray}\label{gj+1/fg}
\left\{
\begin{array}{ll}
\sgn\left(\frac{g_{j+1}(\lambda)}{(f_{j}(\lambda),g_{j}(\lambda))}(\alpha_{j,i})\right)=(-1)^{j-s_{j}+i-1}, &1\leq i\leq s_{j}\\
\sgn\left(\frac{g_{j+1}(\lambda)}{(f_{j}(\lambda),g_{j}(\lambda))}(\beta_{j,i})\right)=(-1)^{j-s_{j}+i}, & 1\leq i\leq s_{j}-1
\end{array}.
\right.
\end{eqnarray}
Now,  $\deg(g_{j+1}(\lambda))=j$ and the leading coefficient of $(f_j(\lambda),g_j(\lambda))$ is positive,  if number  $\theta_{j,1}>\alpha_{j,s_j}$ is sufficiently large, it is evident that  $(f_{j}(\theta_{j,1}),g_{j+1}(\theta_{j,1})) > 0$ and
\begin{eqnarray*}
\sgn\left(\frac{g_{j+1}(\lambda)}{(f_{j}(\lambda),g_{j}(\lambda))}(\theta_{j,1})\right)=(-1)^{\deg(g_{j+1}(\lambda))}=(-1)^{j}.
\end{eqnarray*}
So, each interval in $\{(\alpha_{j,s_j},+\infty), (\alpha_{j,i},\beta_{j,i}),i=1,\cdots,s_j-1\}$ contains  exactly one root of
 $\frac{g_{j+1}(\lambda)}{(f_{j}(\lambda),g_{j}(\lambda))}$, which concludes
 $$
 (f_{j}(\lambda),g_{j+1}(\lambda))=(f_{j}(\lambda),g_{j}(\lambda))\quad \mbox{and}\quad \frac{f_{j}(\lambda)}{(f_{j}(\lambda),g_{j+1}(\lambda))}\ll \frac{g_{j+1}(\lambda)}{(f_{j}(\lambda),g_{j+1}(\lambda))}.$$

 Let $\gamma_{j,1}<\cdots<\gamma_{j,s_j}$ be the roots of $\frac{g_{j+1}(\lambda)}{(f_j(\lambda),g_{j+1}(\lambda))}$.
From the first equation of \dref{LS} and \dref{gj+1/fg}, it follows that
\begin{eqnarray}
\left\{
\begin{array}{ll}
\sgn\left(\frac{f_{j+1}(\lambda)}{(f_{j}(\lambda),g_{j+1}(\lambda))}(\alpha_{j,i})\right)= \sgn\left(\frac{(-\lambda)g_{j+1}(\lambda)}{(f_{j}(\lambda),g_{j+1}(\lambda))}(\alpha_{j,i})\right)= (-1)^{j-s_{j}+i}, &  1\leq i\leq s_{j}\\
\sgn\left(\frac{f_{j+1}(\lambda)}{(f_{j}(\lambda),g_{j+1}(\lambda))}(\gamma_{j,i})\right)= \sgn\left(\frac{f_{j}(\lambda)}{(f_{j}(\lambda),g_{j+1}(\lambda))}(\gamma_{j,i})\right)= (-1)^{j-s_{j}+i},& 1\leq i\leq s_{j}\\
\sgn\left(\frac{f_{j+1}(\lambda)}{(f_{j}(\lambda),g_{j+1}(\lambda))}(0)\right)=\sgn\left(\frac{f_{j}(\lambda)}{(f_{j}(\lambda),g_{j+1}(\lambda))}(0)\right)= (-1)^{j-s_{j}}
\end{array}.
\right.
\end{eqnarray}
Because $\deg(f_{j+1}(\lambda))=j+1$ and the leading coefficient of $(f_j(\lambda),g_j(\lambda))$ is positive, we can choose a sufficiently large $\theta_{j,2}>\gamma_{s_j}>\alpha_{s_j}$  such that $(f_{j}(\theta_{j,2}),g_{j+1}(\theta_{j,2})) > 0$ and
\begin{eqnarray*}
\sgn\left(\frac{f_{j+1}(\lambda)}{(f_{j}(\lambda),g_{j+1}(\lambda))}(\theta_{j,2})\right)=(-1)^{j+1}.
\end{eqnarray*}
Recall that $\frac{f_{j}(\lambda)}{(f_{j}(\lambda),g_{j+1}(\lambda))}\ll \frac{g_{j+1}(\lambda)}{(f_{j}(\lambda),g_{j+1}(\lambda))}$, so each interval in $$\{(0,\alpha_{j,1}),(\gamma_{j,s_j},+\infty), (\gamma_{j,i},\alpha_{j,i+1}), i=1,\ldots,s_j-1\}$$ contains  exactly one root of $\frac{f_{j+1}(\lambda)}{(f_{j}(\lambda),g_{j+1}(\lambda))}$.
Because of this property,  $$(f_{j+1}(\lambda),g_{j+1}(\lambda))=(f_{j}(\lambda),g_{j+1}(\lambda))=(f_{j}(\lambda),g_{j}(\lambda)),$$
 and hence  $\frac{(-\lambda)f_{j}(\lambda)}{(f_{j}(\lambda),g_{j}(\lambda))}\ll \frac{(-\lambda)g_{j+1}(\lambda)}{(f_{j+1}(\lambda),g_{j+1}(\lambda))}\ll \frac{f_{j+1}(\lambda)}{(f_{j+1}(\lambda),g_{j+1}(\lambda))}$.
\end{proof}
It is ready to prove the necessity of Theorem \ref{sta3}. To this end, we introduce some notations. Let $f(\lambda)$ be a
 polynomial whose  roots $\{z_i\}_{i=1}^{p}$ are all real and $z_i< z_j$ for every $i< j$. Denote $\xi(f(\lambda),z_i)$ as  the multiplicity of root $z_i$ and for a real number $\alpha$,  define
 $ \zeta(f(\lambda),\alpha)\triangleq\max\{i\in [1,p]: z_i<\a\}$.

\emph{The proof of the necessity of Theorem \ref{sta3}}: First, in view of \dref{KMBroots}, we know that $\lambda_i, 1\leq i\leq m$ are the $m$  distinct roots of $f_n(\lambda)$ with  multiplicities $t_i$. To proceed the argument,
note that Lemma \ref{ditui} gives $\frac{(-\lambda)g_1(\lambda)}{(f_1(\lambda),g_1(\lambda))}\ll \frac{f_1(\lambda)}{(f_1(\lambda),g_1(\lambda))}$, then by using Lemma \ref{two},
\begin{eqnarray}\label{hs}
\frac{(-\lambda)g_j(\lambda)}{(f_j(\lambda),g_j(\lambda))}\ll \frac{f_j(\lambda)}{(f_j(\lambda),g_j(\lambda))},\quad j=1,\ldots,n.
\end{eqnarray}
Particularly, it turns out that all  the roots of $\frac{f_n(\lambda)}{(f_n(\lambda),g_n(\lambda))}$ are distinct. So $\xi\left(\frac{f_n(\lambda)}{(f_n(\lambda),g_n(\lambda))}, \lambda_i\right)\leq 1$ for all $1\leq i\leq m$ and then
\begin{eqnarray}\label{dym}
\xi\left((f_n(\lambda),g_n(\lambda)), \lambda_i\right)=\xi\left(f_n(\lambda), \lambda_i\right)-\xi\left(\frac{f_n(\lambda)}{(f_n(\lambda),g_n(\lambda))}, \lambda_i\right)\geq t_i-1.
\end{eqnarray}

Now, by \dref{hs} and Lemma \ref{two}, for each $j=1,\ldots,n-1$,
\begin{eqnarray*}
(f_{j+1}(\lambda),g_{j+1}(\lambda)) = \left\{
\begin{array}{ll}
(f_{j}(\lambda),g_{j}(\lambda))(\lambda-\frac{k_{j+1}}{b_{j+1}}), &\mbox{if}~ b_{j+1}\neq 0, (\lambda-\frac{k_{j+1}}{b_{j+1}})|\frac{f_j(\lambda)}{(f_j(\lambda),g_j(\lambda))}\\
(f_{j}(\lambda),g_{j}(\lambda)), &\mbox{otherwise}
\end{array},
\right.
\end{eqnarray*}
which yields that for all  $1\leq i\leq m$,
\begin{eqnarray}\label{dd1}
&&\xi\left((f_{j+1}(\lambda),g_{j+1}(\lambda)), \lambda_i\right)\nonumber\\&=& \left\{
\begin{array}{ll}
\xi\left((f_{j}(\lambda),g_{j}(\lambda)), \lambda_i\right)+1, & \mbox{if}~ b_{j+1}\neq 0, k_{j+1}=\lambda_i b_{j+1},(\lambda-\lambda_i)|\frac{f_j(\lambda)}{(f_j(\lambda),g_j(\lambda))}\\
\xi\left((f_{j}(\lambda),g_{j}(\lambda)), \lambda_i\right), & \mbox{otherwise}
\end{array}.
\right.
\end{eqnarray}
On the other hand, given $j\in [1,n-1]$ and $i\in [1,m]$, if $b_{j+1}\neq 0, k_{j+1}=\lambda_i b_{j+1}$ and $(\lambda-\lambda_i)|\frac{f_j(\lambda)}{(f_j(\lambda),g_j(\lambda))}$, Lemma \ref{two} (i) indicates  $\frac{f_{j+1}(\lambda)}{(f_{j+1}(\lambda),g_{j+1}(\lambda))}\ll \frac{-f_j(\lambda)}{(f_j(\lambda),g_j(\lambda))}$. Moreover,  $\lambda_i$ is a root of  $\frac{f_j(\lambda)}{(f_j(\lambda),g_j(\lambda))}$, it follows that $\zeta\left(\frac{f_{j+1}(\lambda)}{(f_{j+1}(\lambda),g_{j+1}(\lambda))},\lambda_i\right)=\zeta\left(\frac{f_j(\lambda)}{(f_{j}(\lambda),g_{j}(\lambda))}, \lambda_i\right)+1$. Otherwise, by virtue of Lemma \ref{two}, at least one of the following cases will happen:\\
(i) if $b_{j+1}\neq 0, k_{j+1}=\lambda_i b_{j+1}$ and $(k_{j+1}-\lambda b_{j+1})|\frac{f_j(\lambda)}{(f_j(\lambda),g_j(\lambda))}$, then $\frac{f_{j+1}(\lambda)}{(f_{j+1}(\lambda),g_{j+1}(\lambda))}\ll \frac{-f_j(\lambda)}{(f_j(\lambda),g_j(\lambda))};$\\
(ii)  if $b_{j+1}\neq 0$ and $k_{j+1}-\lambda b_{j+1}\nmid \frac{f_j(\lambda)}{(f_j(\lambda),g_j(\lambda))}$, then
 $\frac{f_{j+1}(\lambda)}{(f_{j+1}(\lambda),g_{j+1}(\lambda))}\ll \frac{f_j(\lambda)(k_{j+1}-\lambda b_{j+1})}{(f_j(\lambda),g_j(\lambda))}$;\\
(iii) if $b_{j+1}=0$, then $\frac{(-\lambda)f_{j}(\lambda)}{(f_{j}(\lambda),g_{j}(\lambda))}\ll \frac{f_{j+1}(\lambda)}{(f_{j+1}(\lambda),g_{j+1}(\lambda))}$.\\
All the above three cases lead to $\zeta\left(\frac{f_{j+1}(\lambda)}{(f_{j+1}(\lambda),g_{j+1}(\lambda))},\lambda_i\right)\geq \zeta\left(\frac{f_j(\lambda)}{(f_{j}(\lambda),g_{j}(\lambda))},\lambda_i\right)$. So,
\begin{eqnarray}\label{dd2}
\left\{
\begin{array}{ll}
\zeta\left(\frac{f_{j+1}(\lambda)}{(f_{j+1}(\lambda),g_{j+1}(\lambda))},\lambda_i\right)=\zeta\left(\frac{f_j(\lambda)}{(f_{j}(\lambda),g_{j}(\lambda))}, \lambda_i\right)+1, & \mbox{if case (i) occurs}\\
\zeta\left(\frac{f_{j+1}(\lambda)}{(f_{j+1}(\lambda),g_{j+1}(\lambda))},\lambda_i\right)\geq \zeta\left(\frac{f_j(\lambda)}{(f_{j}(\lambda),g_{j}(\lambda))}, \lambda_i\right),& \mbox{otherwise}
\end{array}.
\right.
\end{eqnarray}
Clearly,  $\zeta\left(\frac{f_1(\lambda)}{(f_{1}(\lambda),g_{1}(\lambda))}, \lambda_i\right)\geq \xi\left((f_{1}(\lambda),g_{1}(\lambda)), \lambda_i\right)=0$. By (\ref{dd1}) and (\ref{dd2}),  it can be derived inductively that
$$
\zeta\left(\frac{f_{n}(\lambda)}{(f_{n}(\lambda),g_{n}(\lambda))}, \lambda_i\right)\geq \xi\left((f_{n}(\lambda),g_{n}(\lambda)), \lambda_i\right).$$
 Together with (\ref{dym}), the above inequality   shows that for each $1\leq i\leq m$,
\begin{eqnarray*}
i-1=\zeta\left(f_{n}(\lambda), \lambda_i\right)\geq\zeta\left(\frac{f_{n}(\lambda)}{(f_{n}(\lambda),g_{n}(\lambda))}, \lambda_i\right)\geq \xi\left((f_{n}(\lambda),g_{n}(\lambda)), \lambda_i\right)\geq t_i-1,
\end{eqnarray*}
which completes the proof.

\section{Proof of sufficiency of Theorem \ref{sta3}.}\label{MR1}
The sufficiency of Theorem \ref{sta3} is  a direct consequence of the following proposition.
\begin{proposition}\label{gd}
Given $m$ real numbers $M_1,\ldots,M_m>0$ and a polynomial $\prod_{i=1}^{m}(\lambda -\lambda_i)^{t_i}$ with $0<\lambda_1<\lambda_2<\cdots<\lambda_{m}$ and $\sum_{i=1}^{m}{t_i}=n$, if
 $t_i\leq i$ for each $i=1,\ldots,m$,  then there exist some $m_j>0$, $k_j>0$, $b_j\geq 0$, $j=1,\ldots,n$ and $m$ distinct indices $i_l, l=1,\ldots,m$ 
  such that $m_{i_l}=M_l$ and
\begin{eqnarray*}
\prod_{i=1}^{m}(\lambda -\lambda_i)^{t_i} \Big| \det(\boldsymbol{K}-\lambda (\boldsymbol{M}+\boldsymbol{B})).
\end{eqnarray*}
\end{proposition}


We now begin the construction of   the required mass-chain system for Proposition \ref{gd}. That is,  to find a sequence of $\lbrace (k_i,b_i,m_i)\rbrace_{i=1}^{n}$  and $m$ indices $i_l$ such that  $\prod_{i=1}^{m}(\lambda-\lambda_{i})^{t_i}|f_n(\lambda)$ and $m_{i_l}=M_l$, $l=1,\ldots,m$, where $f_{n}(\lambda)=\det(\boldsymbol{K}-\lambda (\boldsymbol{M}+\boldsymbol{B}))$. 
Taking account to  \cite[Theorem 4]{chen2018}, we
take  $T=\max_{1\leqslant i\leqslant m}t_i>1$.
Moreover, denote $q_j\triangleq |S_j|$, where
\begin{eqnarray}\label{ST-1}
S_j\triangleq\lbrace \lambda_i: t_i\geq j+1\rbrace,\quad j=1,\ldots, T-1.
\end{eqnarray}
Evidently,  $\sum_{j=1}^{T-1}q_j=n-m$.
We reorder the elements of $S_j$ by $s_{j}(1)<\cdots<s_{j}(q_j)$.

An important observation of Section \ref{n1} is that
(\ref{dym}) implies
 every element in $\bigcup_{j=1}^{T-1}S_j$ is equal to some $k_i/b_i, i\in [2,n]$. This  could help us to design a rule to determine which $k_i/b_i\in \bigcup_{j=1}^{T-1}S_j$. It is the key idea of our proof, so  we offer an example to elaborate it.

 \begin{example}\label{exrule}
Take   $n=15, m=8, t_1=t_5=t_7=t_8=1, t_2=t_4=2, t_3=3, t_6=4$ and $0<\lambda_1<\lambda_2<\cdots<\lambda_8$  in Proposition \ref{gd}. Then, we get sets $S_{j},j=1,2,3$ as shown in  Figure \ref{Fig.3}.  Note that by \dref{KK}--\dref{BB}, for all $i\in[2,15]$,  $-k_i$ and $-b_i$ are located in the secondary diagonals of matrices $\textbf{K}$ and $\textbf{B}$, respectively. We now introduce a $15\times 15$ matrix $\textbf{A}=(a_{ij})$ and assign the elements $a_{i,i+1}$  in the secondary diagonal  the values taken  from  sets $S_{j},j=1,2,3$ (see Figure \ref{Fig.2}). Specifically, for $j=1$, we treat the first $1+q_1=5$ elements of  $a_{i,i+1}$ by
skipping $a_{5,6}$ and letting $a_{i,i+1}=s_1(4-i+1)$ for $i=1,2,3,4$. So, $a_{1,2}=\lambda_6, a_{2,3}=\lambda_4, a_{3,4}=\lambda_3, a_{4,5}=\lambda_2$,  as illustrated in  Fig. 3. Repeat this procedure for elements $a_{i,i+1}, i=\sum_{l=1}^{j-1}q_l+2,\ldots, \sum_{l=1}^{j}q_l+2$ with $j=2, 3$.  For each $i$ with $a_{i,i+1}\in \bigcup_{j=1}^{3}S_j$, we assign $k_{i+1}/b_{i+1}$     a value equivalent to $a_{i,i+1}$.
As a result, $k_2/b_2=k_7/b_7=k_{10}/b_{10}=\lambda_6$, $k_3/b_3=\lambda_4$,  $k_4/b_4=k_{8}/b_{8}=\lambda_3$, $k_5/b_5=\lambda_2$.

\begin{figure}[htbp]
\begin{minipage}[t]{0.4\linewidth}
\centering
\includegraphics[height=4cm,width=6cm]{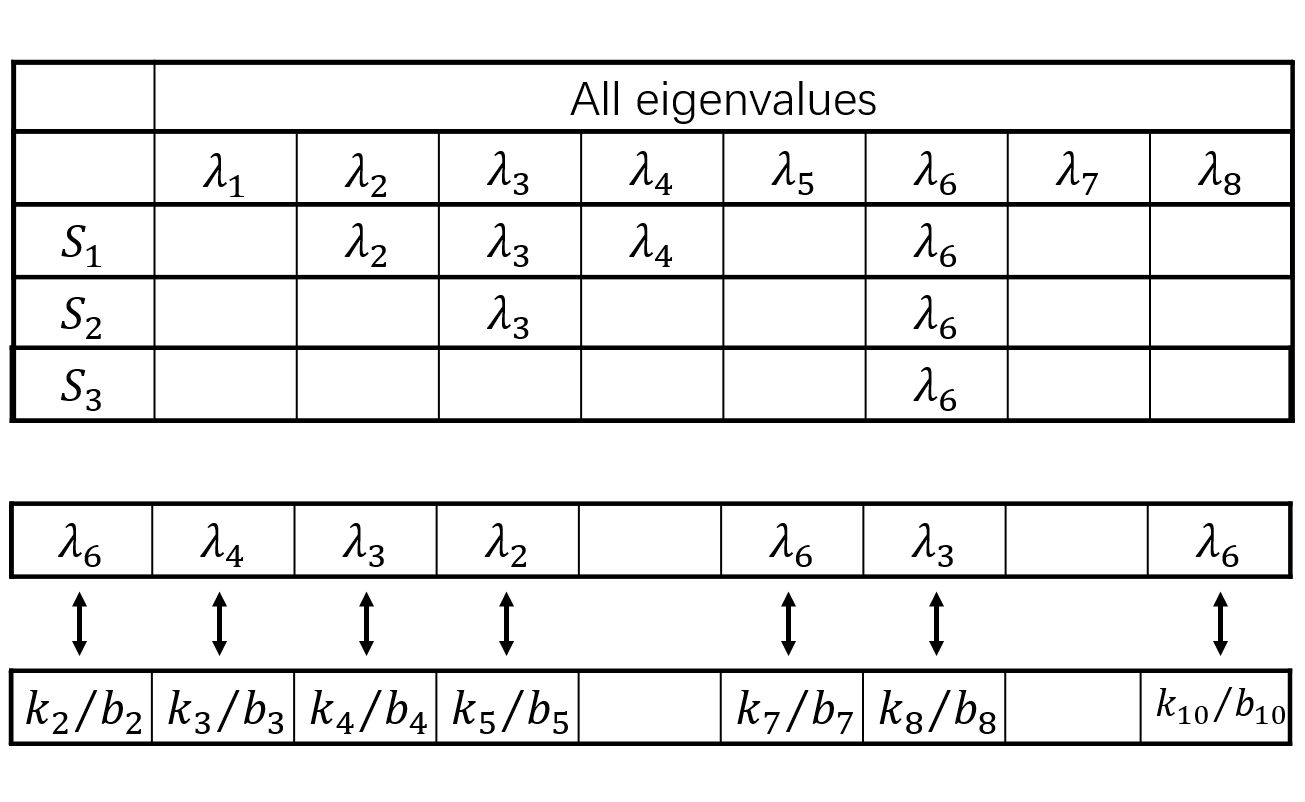}
\caption{Association between $k_i/b_i$ and $\lambda_j$}
\label{Fig.3}
\end{minipage}
\hfill
\begin{minipage}[t]{0.5\linewidth}
\centering
\includegraphics[height=7cm,width=8cm]{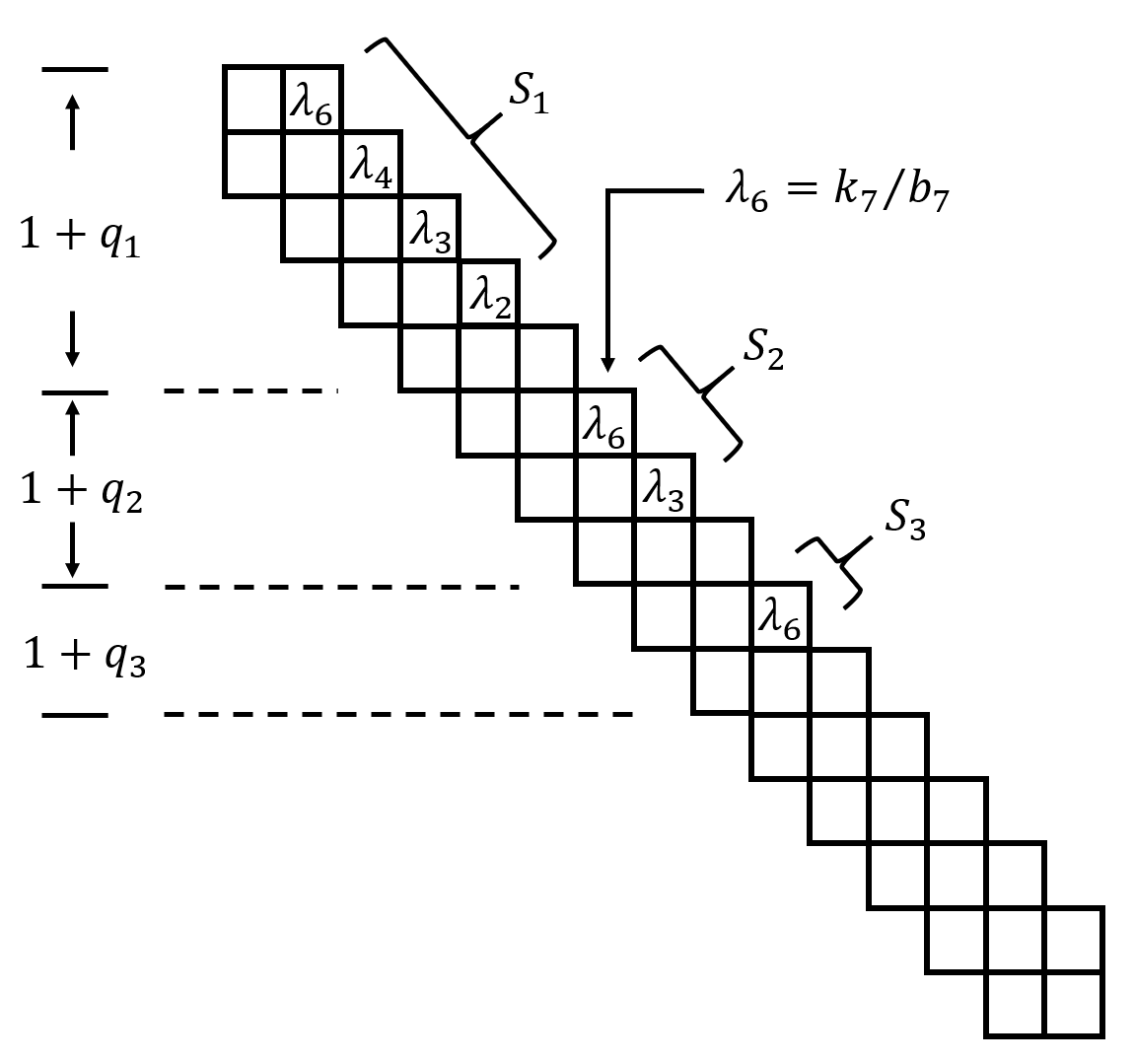}
\caption{Placement of $\lambda_i$ in the secondary diagonal of  matrix $\textbf{A}$}
\label{Fig.2}
\end{minipage}
\end{figure}

\end{example}
In general, the rule to determine $k_i/b_i\in \bigcup_{j=1}^{T-1}S_j$ is summarized as follows:
  \begin{eqnarray}\label{AS}
 \frac{k_i}{b_i} =
s_{j+1}(q_{j+1}-l+1) ,\qquad i=1+j+l,\quad j=0, \ldots, T-2,\quad l=1,\ldots, q_{j+1}.
\end{eqnarray}
The proof of Proposition \ref{gd} thus will be completed  in  three steps.\\
\textit{Step 1}: For each $i=1+j+l$ with $ j=0, \ldots, T-2$ and $ l=1,\ldots, q_{j+1}$,   we  assign $k_i/b_i$ a value taken from $\bigcup_{j=1}^{T-1}S_j$ in the light of \dref{AS}.
So, the cardinal of  set $\lbrace i\in [2,n]: \frac{k_{i}}{b_{i}}\not\in \bigcup_{j=1}^{T-1}S_j \rbrace$ is $(n-1)-(n-m)=m-1$.\\
\textit{Step 2}: Rewrite the elements of $\lbrace 1\rbrace\cup\lbrace i\in [2,n]: \frac{k_{i}}{b_{i}}\not\in \bigcup_{j=1}^{T-1}S_j \rbrace$ by $1=i_1<\cdots<i_{m}$. Let $m_{i_l}=M_{l}$ for  $l=1,\ldots,m$. \\
\textit{Step 3}: Based on the above steps, we compute $(f_i(\lambda),g_i(\lambda))$ for each $i=1,\ldots,n$. Then,  take some suitable parameters $(m_i,b_i,k_i)$, so that if $\frac{k_{i+1}}{b_{i+1}}\not\in \bigcup_{j=1}^{T-1}S_j$,
\begin{eqnarray}\label{fg1}
\left\{
\begin{array}{l}
\frac{f_{i+1}(\lambda)}{(f_{i+1}(\lambda),g_{i+1}(\lambda))}=-\lambda m_{i+1}\frac{g_{i+1}(\lambda)}{(f_{i+1}(\lambda),g_{i+1}(\lambda))}+(k_{i+1}-\lambda b_{i+1})\frac{f_{i}(\lambda)}{(f_{i}(\lambda),g_{i}(\lambda))}\\
\frac{g_{i+1}(\lambda)}{(f_{i+1}(\lambda),g_{i+1}(\lambda))}=\frac{f_{i}(\lambda)}{(f_{i}(\lambda),g_{i}(\lambda))}+(k_{i+1}-\lambda b_{i+1})\frac{g_{i}(\lambda)}{(f_{i}(\lambda),g_{i}(\lambda))}\\
\end{array},
\right.
\end{eqnarray}
otherwise, for  $\frac{k_{i+1}}{b_{i+1}}\in \bigcup_{j=1}^{T-1}S_j$,
\begin{eqnarray}\label{fg2}
\left\{
\begin{array}{l}
\frac{f_{i+1}(\lambda)}{(f_{i+1}(\lambda),g_{i+1}(\lambda))}=-\lambda m_{i+1}\frac{g_{i+1}(\lambda)}{(f_{i+1}(\lambda),g_{i+1}(\lambda))}-b_{i+1}\frac{f_{i}(\lambda)}{(f_{i}(\lambda),g_{i}(\lambda))}\\
(\lambda-\frac{k_{i+1}}{b_{i+1}})\frac{g_{i+1}(\lambda)}{(f_{i+1}(\lambda),g_{i+1}(\lambda))}=\frac{f_{i}(\lambda)}{(f_{i}(\lambda),g_{i}(\lambda))}+(k_{i+1}-\lambda b_{i+1})\frac{g_{i}(\lambda)}{(f_{i}(\lambda),g_{i}(\lambda))}
\end{array}.
\right.
\end{eqnarray}
 The construction of  $\lbrace \frac{f_{i}(\lambda)}{(f_{i}(\lambda),g_{i}(\lambda))},\frac{g_{i}(\lambda)}{(f_{i}(\lambda),g_{i}(\lambda))}\rbrace_{i=1}^{n}$ 
 will be achieved by an induction method
 from $n$ to $1$.

To proceed  the proof, we first derive some technical lemmas.
\begin{lemma}\label{lemma1}
Let $F(\lambda)=\mu\prod_{i=1}^{p}(\lambda-\alpha_i)$ and  $G(\lambda)=\nu\prod_{i=1}^{p}(\lambda-\beta_i)$  be two polynomials of degree $p$ that $G(\lambda)\ll F(\lambda)$, where
$|\mu|>|\nu|$,
$\alpha_1<\cdots<\alpha_p$ and $\beta_1<\cdots<\beta_p$. Then, $F(\lambda)-G(\lambda)$ has $p$ real roots $\gamma_1<\cdots<\gamma_p$ satisfying $\gamma_{i}\in(\alpha_i,\beta_{i+1})$ for $i=1,\ldots,p-1$ and $\gamma_p>\alpha_{p}$. Moreover, the following two statements hold:\\
(i) when $p>1$, for any $\eta\in (0,\min_{1\leqslant i\leqslant p-1}(\alpha_{i+1}-\alpha_i))$, if
\begin{eqnarray}\label{u/v}
\dfrac{|\nu|}{|\mu|}<\ \frac{\min\lbrace 1,\left(\min_{1\leqslant i\leqslant p-1}(\alpha_{i+1}-\alpha_i)-\eta\right)^{p}\rbrace\min\lbrace 1, \eta^p\rbrace}{2+2\max_{j\in[1,p]}\left|\prod_{i=1}^{p}(\alpha_j-\beta_i)\right|}
\end{eqnarray}
 then
\begin{eqnarray}\label{xyu}
\max_{1\leqslant i\leqslant p}(\gamma_i-\alpha_i)<\eta;
\end{eqnarray}
(ii)
when $p=1$, for any $\eta\in (0,\frac{1}{2})$, \dref{xyu} holds provided that
\begin{eqnarray}\label{v/uii}
|\nu|<\frac{|\mu|}{2} \min\left\lbrace \frac{\eta}{\alpha_1-\beta_1},1\right\rbrace.
\end{eqnarray}
\end{lemma}
\begin{proof}
Note that  $G(\lambda)\ll F(\lambda)$ indicates  $\mu\nu>0$ and for each $j=1,\ldots,p-1$,
\begin{eqnarray*}
(F(\alpha_j)-G(\alpha_j))(F(\beta_{j+1})-G(\beta_{j+1}))=-\mu\nu\prod_{i=1}^{p}(\alpha_j-\beta_i)(\beta_{j+1}-\alpha_i)<0,
\end{eqnarray*}
which means $F(\lambda)-G(\lambda)$ has a  root $\gamma_j$ in $(\alpha_j, \beta_{j+1})$. Further, since $\mu\nu>0$ and $|\mu|>|\nu|$,
$
\mu(F(\lambda)-G(\lambda))>0
$
holds for all sufficiently large $\lambda>\alpha_p$.  On the other hand, $$\mu(F(\alpha_p)-G(\alpha_p))=-\mu\nu\prod_{i=1}^{p}(\alpha_p-\beta_i)<0,$$
so $F(\lambda)-G(\lambda)$ has a root $\gamma_p$ in $(\alpha_p, \infty)$. Clearly,  $\gamma_1<\cdots<\gamma_p$.

Next, we show statement (i). Observe that $F(\lambda)-G(\lambda)=(\mu-\nu)\prod_{i=1}^{p}(\lambda-\gamma_{i})$, then for each $j\in[1,p]$,
\begin{eqnarray}\label{aj-gi}
(\mu-\nu)\prod_{i=1}^{p}(\alpha_j-\gamma_i)=F(\alpha_j)-G(\alpha_j)=-\nu\prod_{i=1}^{p}(\alpha_j-\beta_i).
\end{eqnarray}
Let $p>1$. If 
\dref{xyu} fails, denote $l$ as the smallest subscript $i\in [1,p]$ such that $\gamma_i-\alpha_i\geq\eta$. Hence,
$$
0<\alpha_{i+1}-\alpha_i-\eta\leq \alpha_{i+1}-\gamma_i\leq \alpha_l-\gamma_i,\quad i=1,\ldots,l-1.
$$
Moreover, it is clear that $\gamma_i-\alpha_l \geq \gamma_l-\alpha_l>0$ for all $i\geq l$, then by \dref{u/v} and \dref{aj-gi},
\begin{eqnarray*}
&&\left(\min_{1\leqslant i\leqslant p-1}(\alpha_{i+1}-\alpha_i)-\eta\right)^{l-1}(\gamma_{l}-\alpha_l)^{p-l+1}\\
&\leq & \left|\prod_{i=1}^{p}(\alpha_l-\gamma_i)\right|\leq\left|\frac{\nu/\mu}{1-\nu/\mu}\right|\max_{j\in[1,p]}\left|\prod_{i=1}^{p}(\alpha_j-\beta_i)\right|\nonumber\\
&<&2\left|\frac{\nu}{\mu}\right|\max_{j\in[1,p]}\left|\prod_{i=1}^{p}(\alpha_j-\beta_i)\right|<\min\left\lbrace 1,\left(\min_{1\leqslant i\leqslant p-1}(\alpha_{i+1}-\alpha_i)-\eta\right)^{p}\right\rbrace\min\lbrace 1, \eta^p\rbrace\nonumber\\
&\leq &\left(\min_{1\leqslant i\leqslant p-1}(\alpha_{i+1}-\alpha_i)-\eta\right)^{l-1}\eta^{p-l+1},
\end{eqnarray*}
which contradicts to  $\gamma_l-\alpha_l\geq\eta$. So, statement (i) is true.

When $p=1$, \dref{v/uii} and \dref{aj-gi} lead to
\begin{eqnarray}
\gamma_1-\alpha_1\leq \left|\frac{\nu/\mu}{1-\nu/\mu}\right|(\alpha_1-\beta_1)<2\left|\frac{\nu}{\mu}\right|(\alpha_1-\beta_1)<\eta.\nonumber
\end{eqnarray}
The statement (ii) is proved.
\end{proof}

The subsequent parts  focus on  Step 3 of the construction,
whose key idea is to select some appropriate candidates for the roots of $\frac{g_{n}(\lambda)}{(f_{n}(\lambda),g_{n}(\lambda))}$. We set these roots  as  $\lambda_i+\rho_i$, $i=1,\ldots,m-1$, where $\rho_{i}=\varepsilon^{(n+1)^{m-i}}$, 
\begin{eqnarray}\label{vDelta}
\varepsilon=\frac{\Delta^{n^2+n+1}}{n 2^{3(n+1)^3}\Lambda^{(n+1)^2}},\quad \Delta\triangleq \frac{1}{2}\min\lbrace 1,\min_{1\leq i\leq m-1}(\lambda_{i+1}-\lambda_i)\rbrace, \quad \Lambda\triangleq 1+\lambda_m.
\end{eqnarray}
Next, define
\begin{eqnarray}\label{C1C2}
C_1\triangleq\frac{\Delta}{2^{n+1}\Lambda}\quad\mbox{and}\quad
C_2(j)\triangleq\frac{2^{2(n+1)^2}\Lambda^{n+1}}{\Delta^{n}\varepsilon^{(n+1)^{m-j-1}}},\quad j=1,\ldots,m-1,
\end{eqnarray}
as well as
\begin{eqnarray}\label{MC}
 C\triangleq\frac{2^{2n+1}(1+\Lambda^n)}{C_1^{2n^2}\rho_1^{2n}}\quad\mbox{and}\quad  M\triangleq\sum_{k=1}^{m}M_k.
\end{eqnarray}

\begin{remark}\label{constants}
We remark that $\varepsilon<1$,  $\Lambda/\Delta\geq 2$ in \dref{vDelta} and $C>1$ in \dref{MC}. Moreover,
 $C_2(j), j=1,\ldots,m-1$  defined by \dref{C1C2} satisfy (see Appendix \ref{appc})
\begin{eqnarray}\label{ib}
\left\{
\begin{array}{ll}
(1+C_2(m-1))^n \rho_{m-1}<\frac{\Delta}{2},\\
(1+C_2(j))^n \rho_j<(1+C_2(j+1))^n \rho_{j+1},& j=1,\ldots, m-2,\quad m>2\\
\frac{n(1+C_{2}(j-1))^n }{\Delta}\rho_{j-1}<\frac{1}{4}\left(\frac{\Delta^n}{2^{(n+1)^2}\Lambda^{n+1}}\right)\rho_{j+1},& j=2,\ldots,m-2,\quad m>3
\end{array}.
\right.
\end{eqnarray}
\end{remark}

The above series of constants are repeatedly used in the next two lemmas (Lemmas \ref{cas1}--\ref{cas2}), whose
 proofs 
 are contained in  Appendix \ref{appb}. Both the two  lemmas concern the following  polynomials
\begin{eqnarray}\label{FG}
F(\lambda)=\prod_{i=1}^{p}(\lambda-\alpha_i)\quad\mbox{and}\quad G(\lambda)=\prod_{i=1}^{p-1}(\lambda-\beta_i),\qquad p\in [2,m],
\end{eqnarray}
whose roots $\{\a_i\}_{i=1}^p$ and $\{\b_i\}_{i=1}^{p-1}$ satisfy
\begin{eqnarray}\label{aplam}
\alpha_{p}\in[\lambda_p, \bar \lambda)\qquad \mbox{for some number }\bar \lambda>\lambda_p
\end{eqnarray}
and for each $i=1,\ldots,p-1$,
\begin{eqnarray}\label{zcs}
\lambda_{i}+C_{1}^{n}\rho_{i}\leq\alpha_{i}+C_{1}^{n}\rho_i<\beta_{i}<\lambda_{i}+(1+C_2(i))^{n}\rho_i<\lambda_{i+1}.
\end{eqnarray}

\begin{lemma}\label{cas1}
Let $F(\lambda)$ and  $G(\lambda)$ be two polynomials defined by \dref{FG}--\dref{zcs}  with
$\bar \lambda=\Lambda$ in  \dref{aplam}.
For any given constants  $\mu, \nu, m^{*}$ satisfying $0<m^{*}<M$ and
$
-\frac{\mu}{\nu}>CM$, the following two statements hold.\\
(i) If $p>2$, then there exist two monic  polynomials $ F_0(\lambda)$ and $ G_{0}(\lambda)$ with  distinct roots $\alpha_1'<\cdots<\alpha_{p-1}'$ and $\beta_{1}'<\cdots<\beta_{p-2}'$, respectively, satisfying $\alpha_{p-1}'\in(\alpha_{p-1}, \lambda_{p})$ and for $i\in[1,p-2]$,
 \begin{eqnarray}\label{sxb}
\alpha_{i}'\in(\alpha_i,\beta_i),\qquad C_1(\beta_{i}-\alpha_{i})\leq \beta_{i}'-\alpha_{i}'\leq C_2(i)(\beta_{i}-\alpha_{i}).
\end{eqnarray}
 In addition, for some numbers $\lambda^{*},b^{*}>0$ and $\mu_0, \nu_0$ with $-\frac{\mu_0}{\nu_0}>-\frac{\mu}{\nu}-m^{*}$, $ F_0(\lambda)$ and $ G_{0}(\lambda)$ fulfill
\begin{eqnarray}\label{cas1e}
\left\{
\begin{array}{l}
\mu F(\lambda)=- m^{*}\nu \lambda G(\lambda)+b^{*}\mu_0(\lambda^{*}-\lambda) F_0(\lambda)\\
\nu  G(\lambda)=\mu_0 F_0(\lambda)+b^{*}\nu_{0}(\lambda^{*}-\lambda) G_{0}(\lambda)
\end{array}.
\right.
\end{eqnarray}
(ii) If $p=2$, then there are some numbers $\lambda^{*},b^{*}>0$,  $\alpha_{1}'\in(\alpha_{1}, \lambda_{2})$ and  $\mu_0, \nu_0$ with $-\frac{\mu_0}{\nu_0}>-\frac{\mu}{\nu}-m^{*}$  such that polynomials $F_{0}(\lambda)=\lambda-\alpha_{1}'$ and $G_{0}(\lambda)=1$ satisfy equation \dref{cas1e}.
\end{lemma}

\begin{lemma}\label{cas2}
Given  $\mu, \nu, \lambda^{*}$ with $\lambda_p<\lambda^{*}<\Lambda$ and
$
\frac{\mu}{\nu}<0$, the following two statements hold.\\
(i) For polynomials $F(\lambda)$ and  $G(\lambda)$  defined by \dref{FG}--\dref{zcs}  with
$\bar \lambda=\lambda^*$ in  \dref{aplam}, there exist two monic  polynomials $ F_0(\lambda)$ and $ G_{0}(\lambda)$ with  distinct roots $\alpha_1'<\cdots<\alpha_{p}'$ and $\beta_{1}'<\cdots<\beta_{p-1}'$, respectively, such that $\alpha_{p}'=\lambda^{*}$ and \dref{sxb} holds  for all $i\in[1,p-1]$. In addition, for some numbers  $m^{*},b^{*}>0$ and $\mu_0, \nu_0$ with $-\frac{\mu_0}{\nu_0}>-\frac{\lambda_1}{\Lambda}\frac{\mu}{\nu}$, $F_0(\lambda)$ and $ G_{0}(\lambda)$ fulfill
\begin{eqnarray}\label{cas2e}
\left\{
\begin{array}{l}
\mu F(\lambda)=- m^{*}\nu\lambda G(\lambda)-b^{*}\mu_0 F_0(\lambda)\\
\nu (\lambda-\lambda^{*})G(\lambda)=\mu_0 F_0(\lambda)+b^{*}\nu_{0}(\lambda^{*}-\lambda) G_{0}(\lambda)\\
\end{array}.
\right.
\end{eqnarray}
(ii) For $F(\lambda)=\lambda-\alpha_1$ with $\alpha_{1}<\lambda^{*}$  and $G(\lambda)=1$, there are  some numbers $m^{*},b^{*}>0$ and $\mu_0, \nu_0$ with  $-\frac{\mu_0}{\nu_0}>-\frac{\lambda_1}{\Lambda}\frac{\mu}{\nu}$ such that polynomials  $ F_0(\lambda)=\lambda-\lambda^{*}$ and $ G_{0}(\lambda)=1$ satisfy equation \dref{cas2e}.
\end{lemma}

\begin{lemma}\label{lemma21}
Let  $F_n(\lambda)=\prod_{i=1}^{m}(\lambda-\lambda_i)$ and  $G_n(\lambda)=\prod_{i=1}^{m-1}(\lambda-\lambda_i-\rho_i)$ be two polynomials  and $\mu_n, \nu_n$ be two numbers satisfying
\begin{eqnarray}\label{munvn}
-\frac{\mu_n}{\nu_n}>C\left(\frac{\Lambda}{\lambda_1}\right)^{n}\frac{\Lambda M}{\Lambda-\lambda_1}.
\end{eqnarray}
Then, there exist some monic polynomials $\lbrace F_j(\lambda)\rbrace_{j=1}^{n-1}$, $\lbrace G_{j}(\lambda)\rbrace_{j=1}^{n-1}$ and some  sequences of numbers $\lbrace (\lambda_{j}^{*}, b_{j}, m_j)\rbrace_{j=2}^{n}$, $\lbrace (\mu_j,\nu_j)\rbrace_{j=1}^{n-1}$ such that for each $j\in [1,n-1]$, the following two properties hold:\\
(i) $\lambda_{j+1}^{*}, b_{j+1}, m_{j+1}>0$ and  $-\frac{\mu_j}{\nu_{j}}>C(\frac{\Lambda}{\lambda_1})^j\frac{\Lambda M}{\Lambda-\lambda_1}$;\\
(ii) if $j\in [1,n-1]\setminus \bigcup_{h=0}^{T-2}[l+1+\sum_{h=1}^{l}q_{h},l+\sum_{h=1}^{l+1}q_{h}]$, then
\begin{eqnarray}\label{dt1}
\left\{
\begin{array}{l}
\mu_{j+1}F_{j+1}(\lambda)=- m_{j+1}\nu_{j+1}\lambda G_{j+1}(\lambda)+b_{j+1}\mu_j (\lambda_{j+1}^{*}-\lambda) F_j(\lambda)\\
\nu_{j+1}G_{j+1}(\lambda)=\mu_j F_j(\lambda)+b_{j+1}\nu_{j} (\lambda_{j+1}^{*}-\lambda) G_{j}(\lambda)
\end{array},
\right.
\end{eqnarray}
otherwise, for $j\in \bigcup_{l=0}^{T-2}[l+1+\sum_{h=1}^{l}q_{h},l+\sum_{h=1}^{l+1}q_{h}]$,
\begin{eqnarray}\label{dt2}
\left\{
\begin{array}{l}
\mu_{j+1}F_{j+1}(\lambda)=- m_{j+1}\nu_{j+1} \lambda G_{j+1}(\lambda)-b_{j+1}\mu_j F_j(\lambda)\\
\nu_{j+1}(\lambda-\lambda_{j+1}^{*})G_{j+1}(\lambda)=\mu_j  F_j (\lambda)+b_{j+1}\nu_{j}(\lambda_{j+1}^{*}-\lambda)G_{j}(\lambda)\\
\lambda_{j+1}^{*}=s_{l+1}(l+1+\sum_{h=1}^{l+1}q_{h}-j)
\end{array},
\right.
\end{eqnarray}
where $s_{l+1}(l+1+\sum_{h=1}^{l+1}q_{h}-j)\in S_{l+1}$ and $S_{l+1}$ is  defined by \dref{ST-1}.
\end{lemma}

\begin{proof}
For $j=n-1, \ldots,1$, we construct a series of numbers $\lambda_{j+1}^{*}, b_{j+1}, m_{j+1}$, $\mu_j ,\nu_j$ and  polynomials  $F_{j}(\lambda)$, $G_{j}(\lambda)$ on the basis of $\mu_{j+1} ,\nu_{j+1}$, and $F_{j+1}(\lambda)$, $G_{j+1}(\lambda)$,  according to the following strategies:
\begin{enumerate}
\item[a)] if $j\in [1,n-1]\setminus \bigcup_{l=0}^{T-2}[l+1+\sum_{h=1}^{l}q_{h},l+\sum_{h=1}^{l+1}q_{h}]$, we apply Lemma \ref{cas1} with $F(\lambda)=F_{j+1}(\lambda)$, $G(\lambda)=G_{j+1}(\lambda)$, $m^{*}=m_{j+1}$, $\mu=\mu_{j+1}$, $\nu=\nu_{j+1}$ to obtain $\lambda_{j+1}^{*}, b_{j+1}$, $\mu_j ,\nu_j$ and $F_{j}(\lambda)$, $G_{j}(\lambda)$, where
    \begin{eqnarray}\label{mj+1}
m_{j+1}\triangleq\left\{
\begin{array}{ll}
M_{l+2},& j=l+1+\sum_{h=1}^{l+1}q_{h}\,\,\mbox{for}\,\, l<T-2\\
M_{j+1-n+m
},& j>T-2+n-m
\end{array};
\right.
\end{eqnarray}

\item[b)] if $j\in [l+1+\sum_{h=1}^{l}q_{h},l+\sum_{h=1}^{l+1}q_{h}]$ for some $l\in[0,T-2]$, we set $\lambda_{j+1}^{*}=s_{l+1}(l+1+\sum_{h=1}^{l+1}q_{h}-j)$  and apply Lemma \ref{cas2} with $F(\lambda)=F_{j+1}(\lambda)$, $G(\lambda)=G_{j+1}(\lambda)$, $\lambda^{*}=\lambda_{j+1}^{*}$, $\mu=\mu_{j+1}$, $\nu=\nu_{j+1}$ to obtain $m_{j+1}, b_{j+1}$, $\mu_j ,\nu_j$, and $F_{j}(\lambda)$, $G_{j}(\lambda)$.
\end{enumerate}

We shall use the  induction method to show that  either  strategy a) or strategy b) can be implemented for each $j=n-1,\ldots,1$.
First, let  $j=n-1$. Observe that $T\leq m$, it is easy to compute
$$
(T-2)+\sum_{h=1}^{l+1}q_{h}\leq (m-2)+(n-m)=n-2.
$$
Hence, $n-1\notin \bigcup_{l=0}^{T-2}[l+1+\sum_{h=1}^{l}q_{h},l+\sum_{h=1}^{l+1}q_{h}]$.
In addition, since \dref{munvn} and  \dref{mj+1} yield $-\frac{\mu_n}{\nu_n}>CM$ and  $m_n\in (0,M)$, by applying   Lemma \ref{cas1}(i)  with $F(\lambda)=F_n(\lambda)$, $G(\lambda)=G_n(\lambda)$, $m^{*}=m_n$, $\mu=\mu_n$, $\nu=\nu_n$, we can find some numbers  $\lambda_{n}^{*}, b_{n}, \mu_{n-1},\nu_{n-1}$ with  $\lambda_{n}^{*},b_{n}>0$ and
$$-\frac{\mu_{n-1}}{\nu_{n-1}}>-\frac{\mu_{n}}{\nu_{n}}-m_n>C\left(\frac{\Lambda}{\lambda_1}\right)^{n-1}\frac{\Lambda M}{\Lambda-\lambda_1},$$
and two monic polynomials $F_{n-1}(\lambda)$, $G_{n-1}(\lambda)$ such that \dref{dt1} holds. So, strategy a) applies and   both (i) and (ii) are true for these
$\lambda_{n}^{*}, b_{n}, \mu_{n-1},\nu_{n-1}, F_{n-1}(\lambda), G_{n-1}(\lambda)$.

Now, assume that we have constructed the required $\lbrace (\lambda_{j}^{*}, b_{j}, m_j)\rbrace_{j=n-r+1}^{n}$, $\lbrace (\mu_j,\nu_j)\rbrace_{j=n-r}^{n-1}$,   $\lbrace F_j(\lambda)\rbrace_{j=n-r}^{n-1}$ and $\lbrace G_{j}(\lambda)\rbrace_{j=n-r}^{n-1}$ for some $r\in [1,n-2]$ by following either strategy a) or strategy b), so that
 properties  (i) and (ii) hold for $j=n-1, \ldots, n-r$. 
 Considering Lemmas \ref{cas1} and \ref{cas2},
we write $F_{n-j}(\lambda)=\prod_{i=1}^{z_{n-j}}(\lambda-\alpha_{n-j}(i))$ with $\alpha_{n-j}(1)<\cdots<\alpha_{n-j}(z_{n-j})$ and $G_{n-j}(\lambda)=\prod_{i=1}^{z_{n-j}-1}(\lambda-\beta_{n-j}(i))$ with $\beta_{n-j}(1)<\cdots<\beta_{n-j}(z_{n-j}-1)$, $j=0,1,\ldots,r$. 
Here, for each $ j= 0,\ldots,r-1$,
\begin{eqnarray}\label{deg}
z_{n-j-1}=\left\{
\begin{array}{ll}
z_{n-j}-1,  &n-j-1\not\in \cup_{l=0}^{T-2}[l+1+\sum_{h=1}^{l}q_{h},l+\sum_{h=1}^{l+1}q_{h}]\\
z_{n-j},&  n-j-1\in \cup_{l=0}^{T-2}[l+1+\sum_{h=1}^{l}q_{h},l+\sum_{h=1}^{l+1}q_{h}]\\
\end{array}.
\right.
\end{eqnarray}
Furthermore, if $z_{n-r}>1$, then for each   $j\in[1,r]$ and     $i\in[1,z_{n-r}-1]$,
\begin{eqnarray}\label{2abC1C2}
\left\{
\begin{array}{ll}
\alpha_{n-j}(i)\in(\alpha_{n-j+1}(i),\beta_{n-j+1}(i))\\
C_1(\beta_{n-j+1}(i)-\alpha_{n-j+1}(i))\leq \beta_{n-j}(i)-\alpha_{n-j}(i)\leq C_2(i)(\beta_{n-j+1}(i)-\alpha_{n-j+1}(i))
\end{array}.
\right.
\end{eqnarray}
Recall that $\a_n(i)=\lambda_i$ and $\beta_{n}(i)=\lambda_i+\rho_i$, \dref{2abC1C2}  implies that for  $i\in[1,z_{n-r}-1]$,
\begin{eqnarray}\label{wxf}
\beta_{n-r}(i)-\alpha_{n-r}(i)\geq C_1^{r}(\beta_{n}(i)-\alpha_{n}(i))=C_1^{r}\rho_{i}.
\end{eqnarray}
Moreover, by \dref{2abC1C2} again,
\begin{eqnarray}
\beta_{n-j}(i)-\alpha_n(i)&=&\beta_{n-j}(i)-\alpha_{n-j}(i)+\alpha_{n-j}(i)-\alpha_n(i)\nonumber\\
&\leq& C_2(i)(\beta_{n-j+1}(i)-\alpha_{n-j+1}(i))+(\beta_{n-j+1}(i)-\alpha_n(i))\nonumber\\
&\leq &(1+C_2(i))(\beta_{n-j+1}(i)-\alpha_n(i)),\nonumber
\end{eqnarray}
then a straightforward calculation leads to
\begin{eqnarray}\label{abn-r<}
\alpha_{n}(i)<\alpha_{n-r}(i)<\beta_{n-r}(i)<\alpha_{n}(i)+(1+C_2(i))^{r}(\beta_{n}(i)-\alpha_{n}(i)).
\end{eqnarray}
Note that by \dref{ib} in Remark \ref{constants},
$$(1+C_2(i))^{r}(\beta_{n}(i)-\alpha_{n}(i))<(1+C_2(i))^{n}\rho_i<\Delta,$$
so  by virtue of  \dref{wxf} and \dref{abn-r<},
\begin{eqnarray}\label{jb}
\lambda_{i}+C_1^{n}(i)\rho_{i}<\alpha_{n-r}(i)+C_1^{n}(i)\rho_{i}<\beta_{n-r}(i)<\lambda_{i}+(1+C_2(i))^{n}\rho_i<
\lambda_{i+1}.
\end{eqnarray}
We remark that no matter strategy a) or  b) applies, it always infers  
\begin{eqnarray*}
\left\{
\begin{array}{l}
\alpha_{n-r}(z_{n-r})<\beta_{n-r}(z_{n-r})<\lambda_{z_{n-r}+1},\quad z_{n-r+1}-1=z_{n-r}\\
\alpha_{n-r}(z_{n-r})=\lambda_{n-r+1}^{*}<\Lambda,\quad ~~~~\quad~~~~~ z_{n-r+1}=z_{n-r}\\
\end{array},
\right.
\end{eqnarray*}
and hence
\begin{eqnarray}\label{jb2}
\alpha_{n-r}(z_{n-r})\in [\lambda_{z_{n-r}},\Lambda).
\end{eqnarray}

We now  verify that at least one of the strategies a) and b) is valid for $j=n-r-1$. It is discussed by two cases.\\\\
\textbf{Case 1:} $n-r-1\not\in \cup_{l=0}^{T-2}[l+1+\sum_{h=1}^{l}q_{h},l+\sum_{h=1}^{l+1}q_{h}]$.
Because of \dref{deg}, we  estimate $z_{n-r}$ directly by
\begin{eqnarray}
z_{n-r}\geq m-\left(n-2-\sum_{h=1}^{T-1}q_h\right)=m-\left(n-2-(n-m)\right)=2.\nonumber
\end{eqnarray}
Now, $-\frac{\mu_{n-r}}{\nu_{n-r}}>C(\frac{\Lambda}{\lambda_1})^{n-r}\frac{\Lambda M}{\Lambda-\lambda_1}$, it thus gives $-\frac{\mu_{n-r}}{\nu_{n-r}}>CM$. So combining \dref{jb} and \dref{jb2}, it shows that the assumptions of Lemma \ref{cas1} are fulfilled. Therefore,  Lemma \ref{cas1} is applicable and   strategy a) works. 
We  thus conclude properties (i) and (ii) for $j=n-r-1$ by Lemma \ref{cas1} and
 \begin{eqnarray*}
-\frac{\mu_{n-r-1}}{\nu_{n-r-1}}>-\frac{\mu_{{n-r}}}{\nu_{n-r}}-m_{n-r}>C\left(\frac{\Lambda}{\lambda_1}\right)^{n-r}\frac{\Lambda M}{\Lambda-\lambda_1}-M>C\left(\frac{\Lambda}{\lambda_1}\right)^{n-r-1}\frac{\Lambda M}{\Lambda-\lambda_1}.
\end{eqnarray*}
\textbf{Case 2:} $n-r-1\in [l+1+\sum_{h=1}^{l}q_{h},l+\sum_{h=1}^{l+1}q_{h}]$ for some $l\in[0,T-2]$. If $n-r-1<l+\sum_{h=1}^{l+1}q_{h}$, then $n-r\in [l+1+\sum_{h=1}^{l}q_{h},l+\sum_{h=1}^{l+1}q_{h}]$, and hence $F_{n-r}(\lambda)$ is constructed by strategy b). In view of Lemma \ref{cas2},  the maximal root of $F_{n-r}(\lambda)$ is equivalent to $\lambda_{n-r+1}^{*}$, which shows
\begin{eqnarray}
s_{l+1}(l+2+\sum\nolimits_{h=1}^{l+1}q_{h}-n+r)>s_{l+1}(l+1+\sum\nolimits_{h=1}^{l+1}q_{h}-n+r)=\lambda_{n-r+1}^{*}=\alpha_{n-r}(z_{n-r}).\nonumber
\end{eqnarray}

If $n-r-1=l+\sum_{h=1}^{l+1}q_{h}$, then $F_{n-r}(\lambda)$ is constructed by strategy a). Observe that
\begin{eqnarray}
z_{n-r}=T-1-(T-l-2)=l+1,\nonumber
\end{eqnarray}
so by \dref{jb}, $\alpha_{n-r}(z_{n-r})<\lambda_{z_{n-r}+1}=\lambda_{l+2}$. On the other hand, if we
suppose $s_{l+1}(l+2+\sum_{h=1}^{l+1}q_{h}-n+r)=\lambda_{i}$ for some $i\in[1,m]$, then   $i\geq t_i\geq l+2$ because of \dref{ST-1}. As a consequence,
\begin{eqnarray}\label{lj2}
s_{l+1}(l+2+\sum\nolimits_{h=1}^{l+1}q_{h}-n+r)\geq\lambda_{l+2},
\end{eqnarray}
and hence
\begin{eqnarray}\label{gdy}
s_{l+1}(l+2+\sum\nolimits_{h=1}^{l+1}q_{h}-n+r)>\alpha_{n-r}(z_{n-r}).
\end{eqnarray}
We thus conclude  \dref{gdy} is always true in strategy a). By \dref{jb2},
$$
\alpha_{n-r}(z_{n-r}) \in [ \lambda_{z_{n-r}}, s_{l+1}(l+2+\sum\nolimits_{h=1}^{l+1}q_{h}-n+r)).
$$
 This combining  with  \dref{jb} and \dref{jb2} verifies the  assumptions of Lemma \ref{cas2}. Hence,  Lemma \ref{cas2} applies and   strategy b) can be implemented.
So,  properties (i) and (ii) hold for $j=n-r-1$ due to Lemma \ref{cas2} and
 $$-\frac{\mu_{n-r-1}}{\nu_{n-r-1}}>-\frac{\mu_{n-r}}{\nu_{n-r}}\frac{\lambda_1}{\Lambda}>C\left(\frac{\Lambda}{\lambda_1}\right)^{n-r-1}\frac{\Lambda M}{\Lambda-\lambda_1}.$$
 The induction is completed.
\end{proof}

\begin{proof}[Proof of Proposition \ref{gd}]
Let $\{\mu_n,\nu_n,F_n(\lambda), G_n(\lambda)\}$ be defined in Lemma \ref{lemma21}, then
we can construct a series of numbers $\lbrace (\lambda_{j}^{*}, b_{j}, m_j)\rbrace_{j=2}^{n}$, $\lbrace (\mu_j,\nu_j)\rbrace_{j=1}^{n-1}$ and some monic polynomials
$\lbrace F_j(\lambda)\rbrace_{j=1}^{n-1}$, $\lbrace G_{j}(\lambda)\rbrace_{j=1}^{n-1}$. 
Note that $-\frac{\mu_1}{\nu_1}>M>M_1$, set
\begin{eqnarray}\label{bmk1}
\left\{
\begin{array}{l}
m_1=M_1,\quad b_1=-\frac{\mu_1}{\nu_1}-M_1,\quad k_1=-\alpha_1(1)\frac{\mu_1}{\nu_1}\\
k_j=\lambda_{j}^{*}b_j,\quad j=2,\ldots,n\\
\end{array}.
\right.
\end{eqnarray}
We shall see that $\lbrace(k_j,b_j,m_j)\rbrace_{j=1}^{n}$ meet the requirement of Proposition \ref{gd}.

In fact, define a sequence of polynomials $\lbrace D_{j}(\lambda)\rbrace_{j=1}^{n}$ as follows:
\begin{eqnarray}\label{gcd}
\left\{
\begin{array}{ll}
D_1(\lambda)=1\\
D_{k+1}(\lambda)=D_{k}(\lambda),& k\in [1,n-1]\setminus \bigcup_{l=0}^{T-2}[l+1+\sum_{h=1}^{l}q_{j},l+\sum_{h=1}^{l+1}q_{j}]\\
D_{k+1}(\lambda)=(\lambda-\lambda_{k+1}^{*})D_{k}(\lambda),& k\in \bigcup_{l=0}^{T-2}[l+1+\sum_{h=1}^{l}q_{j},l+\sum_{h=1}^{l+1}q_{j}]
\end{array},
\right.\nonumber
\end{eqnarray}
and let
\begin{eqnarray}\label{fgFG}
f_{j}(\lambda)=\frac{\mu_j}{\nu_1}D_{j}(\lambda)F_{j}(\lambda)\quad\mbox{and}\quad  g_{j}(\lambda)=\frac{\nu_j}{\nu_1}D_{j}(\lambda)G_{j}(\lambda),\quad j=1,\ldots,n.
\end{eqnarray}
Clearly,
$
f_n(\lambda)=\frac{\mu_n}{\nu_1}D_{n}(\lambda)F_{n}(\lambda)=\frac{\mu_n}{\nu_1}\prod_{j=1}^{m}(\lambda-\lambda_j)^{t_j}.
$
The rest of the proof is to check whether $\lbrace f_{j}(\lambda)\rbrace_{j=1}^n$ satisfy the recursive formula in Lemma  \ref{ditui}.

Firstly, \dref{bmk1} and \dref{fgFG} imply $f_1(\lambda)=k_1-\lambda(m_1+b_1)$. Moreover, in view of Lemma \ref{lemma21},  $G_{1}(\lambda)=1$, which
 indicates  $g_{1}(\lambda)=1$. Next, by Lemma \ref{lemma21} again,
 for $j\in [1,n-1]\setminus \bigcup_{l=0}^{T-2}[l+1+\sum_{h=1}^{l}q_{h},l+\sum_{h=1}^{l+1}q_{h}]$,
\begin{eqnarray}
\left\{
\begin{array}{l}
f_{j+1}(\lambda)=- m_{j+1}\lambda g_{j+1}(\lambda)+(k_{j+1}-\lambda b_{j+1})f_j(\lambda)\\
g_{j+1}(\lambda)=f_j(\lambda)+(k_{j+1}-\lambda b_{j+1})g_{j}(\lambda)
\end{array},
\right. \nonumber
\end{eqnarray}
and for $j\in \bigcup_{l=0}^{T-2}[l+1+\sum_{h=1}^{l}q_{h},l+\sum_{h=1}^{l+1}q_{h}]$, we compute
\begin{eqnarray}
f_{j+1}(\lambda)&=&\frac{1}{\nu_1}\mu_{j+1} D_{j+1}(\lambda)F_{j+1}(\lambda)=
\frac{1}{\nu_1} D_{j+1}(\lambda)(- m_{j+1}\nu_{j+1}\lambda G_{j+1}(\lambda)-b_{j+1}\mu_j F_j(\lambda))\nonumber\\
&=&- m_{j+1} \lambda g_{j+1}(\lambda)-\frac{(\lambda-\lambda_{j+1}^{*})D_{j}(\lambda)}{\nu_1}b_{j+1}\mu_j F_j(\lambda)\nonumber\\
&=&- m_{j+1} \lambda g_{j+1}(\lambda)+(k_{j+1}-\lambda b_{j+1})f_j(\lambda),\nonumber\\
g_{j+1}(\lambda)&=&\frac{\nu_{j+1}}{\nu_1}D_{j+1}(\lambda)G_{j+1}(\lambda)=\frac{D_{j}(\lambda)}{\nu_1}\nu_{j+1} (\lambda-\lambda_{j+1}^{*}) G_{j+1}(\lambda)\nonumber\\
&=&\frac{D_{j}(\lambda)}{\nu_1}(\mu_j F_j(\lambda)+b_{j+1}\nu_{j} (\lambda_{j+1}^{*}-\lambda)
G_{j}(\lambda))=f_j(\lambda)+(k_{j+1}-\lambda b_{j+1})g_{j}(\lambda).\nonumber
\end{eqnarray}
The assertion thus follows.
\end{proof}
\section{Concluding remarks.}\label{conre}
The  emergence of inerters in engineering  brings some new phenomena in the study of inverse  problems. Particularly, it enables a mass-chain system to possess multiple eigenvalues.
This paper has solved the IEP for the ``fixed-free'' case, where the real numbers for eigenvalue  assignment can be taken arbitrarily  positive. 
Another common situation is that the  both ends of the system are fixed at the wall. For such ``fixed-fixed'' systems, the construction cannot follow readily from the method developed in Section \ref{MR1}. To address this issue,  a more detailed analysis is required and it would be our next work.


\appendix
\section{Proof of Example \ref{thm3}}\label{appa}
For each $j=1,\ldots,5$, let $F_j(\lambda)$ and $ G_j(\lambda)$ be two monic polynomials such that
\begin{eqnarray*}
\mu_j F_j(\lambda)=\frac{f_j(\lambda)}{\left(f_j(\lambda), g_j(\lambda)\right)}\quad\mbox{and}\quad \nu_j G_j(\lambda)=\frac{g_j(\lambda)}{\left(f_j(\lambda),g_j(\lambda)\right)},
\end{eqnarray*}
where $\mu_j$ and $\nu_j$ are the leading coefficients of $f_j(\lambda)$ and $g_j(\lambda)$, respectively. Apparently, $\mu_j\nu_j<0$ and
\begin{eqnarray}\label{FjGjdef}
 (F_{j}(\lambda),G_{j}(\lambda))=1,\quad j=1,\ldots,5.
\end{eqnarray}
Moreover, denote $\alpha_{j}(1)<\cdots<\alpha_{j}(s_j)$ and $\beta_{j}(1)<\cdots<\beta_{j}(s_j-1)$ as the roots of $F_{j}(\lambda)$ and $G_j(\lambda)$.
Define
$$
J\triangleq\lbrace j:  k_j=\lambda_3 b_j, j=2,3,4,5\rbrace
\quad \mbox{
and}\quad   H\triangleq\lbrace j: F_{j-1}(\lambda_3)=0, j=2,3,4,5\rbrace.
$$
 We first assert that $J\cap H$ cannot contain any adjacent natural numbers. Otherwise, suppose there is a number $i\in\lbrace 1,2,3\rbrace$ such that
$i+1, i+2\in J\cap H$. So, 
$k_{i+1}/b_{i+1}=\lambda_3$ and by Lemma \ref{two},
\begin{eqnarray}\label{Fj+1Gj+1}
\mu_{i+1}F_{i+1}(\lambda)=- m_{i+1}\nu_{i+1} \lambda G_{i+1}(\lambda)-b_{i+1}\mu_i F_i(\lambda).
\end{eqnarray}
Since the definition of $H$ indicates $F_i(\lambda_3)=F_{i+1}(\lambda_3)=0$, then by \dref{Fj+1Gj+1},
$$
m_{i+1}\nu_{i+1} \lambda_3  G_{i+1}(\lambda_3)=-b_{i+1}\mu_i F_i(\lambda_3)-\mu_{i+1}F_{i+1}(\lambda_3)=0.$$
Hence, $(\lambda-\lambda_3)|(F_{i+1}(\lambda),G_{i+1}(\lambda))$, which contradicts to \dref{FjGjdef} that $(F_{i+1}(\lambda),G_{i+1}(\lambda))=1$.

We in fact have derived $|J\cap H|\leq 2$,  which together with Lemma \ref{one} implies $F(\lambda_3)=0$ and $5\not\in H$.  On the other hand,  \dref{dym} and \dref{dd1} in Section \ref{n1} infer  $|J\cap H|\geq 2$, and hence $|J\cap H|=2$.  Observe that $J\cap H\not=\lbrace 2,3\rbrace$ or $\lbrace 3,4\rbrace $, it then follows that $J\cap H=\lbrace 2,4\rbrace$.   So,  $\a_1(1)=\lambda_3$ and by Lemma \ref{two}(i), $s_2=1$ and  $\a_2(1)<\lambda_3$. For the roots of $F_j(\lambda),  j=3,4,5$, they can be discussed similarly by  using Lemma \ref{two} repeatedly. Indeed,  $3\notin J\cap H$ and $4\in  H$ lead to
$s_3=2$ and $F_3(\lambda_3)=0$, respectively. So, 
\begin{eqnarray*}\label{a3<}
\left\{
\begin{array}{ll}
0<\a_3(1)<\a_2(1)<\a_3(2)=\lambda_3,& b_3=0\\
\a_3(1)<\min\{\a_2(1),k_3/b_3\}<\lambda_3,\quad \mbox{and}\quad \a_3(2)=\lambda_3<k_3/b_3,& b_3>0
\end{array}.
\right.
\end{eqnarray*}
In addition,  $4\in J\cap H$ indicates $s_4=2$ and
\begin{eqnarray}\label{a4<}
\a_4(1)<\a_3(1)< \a_4(2)<\a_3(2)=\lambda_3.
\end{eqnarray}
At last, since $5\notin J\cap H$,  $s_5=3$. By $|J\cap H|=2$ and \dref{dd1},
$(f_5(\lambda),g_5(\lambda))=(\lambda-\lambda_3)^2,$
which immediately gives $F_5(\lambda)=\prod_{i=1}^3(\lambda-\lambda_i)$. So, $\a_5(1)=\lambda_1<\a_5(2)=\lambda_2<\a_5(3)=\lambda_3$. If $b_5>0$, then  \dref{a4<} infers that $\lambda_3<k_5/b_5$. Therefore,
\begin{eqnarray*}\label{a5<}
\left\{
\begin{array}{ll}
0<\a_5(1)<\a_4(1)<\a_5(2)<\a_4(2)<\a_5(3)=\lambda_3,& b_5=0\\
\a_5(1)<\a_4(1)<\a_5(2)<\a_4(2)<\a_5(3)=\lambda_3<k_5/b_5,& b_5>0
\end{array}.
\right.
\end{eqnarray*}
 We thus summarize
 \begin{eqnarray}\label{ufb3}
\left\{
\begin{array}{l}
s_5=3,\quad s_{4}=s_3=2,\quad s_2=s_1=1\\
\alpha_{2}(s_2),\alpha_{4}(s_4)<\lambda_3\quad \mbox{and}\quad  \alpha_{1}(s_1)=\alpha_{3}(s_3)=\alpha_{5}(s_5)=\lambda_3\\
\alpha_{j+1}(i)<\alpha_{j}(i)<\alpha_{j+1}(i+1),\quad j\in[1,4], \, i\in[1,s_{j+1}-1]
\end{array}
\right.
\end{eqnarray}
and
 \begin{eqnarray}\label{kb>lam3}
\left\{
\begin{array}{ll}
k_j/b_j=\lambda_3\,\,\mbox{and}\,\,b_j>0,&j=2,4\\
k_j/b_j>\lambda_3 \,\,\mbox{if}\,\,b_j>0,&j=3,5
\end{array}.
\right.
\end{eqnarray}
This means  $J=J\cap H=\lbrace 2,4\rbrace$ and as a consequence,  by Lemma \ref{two},
\begin{eqnarray}\label{ufb}
\left\{
\begin{array}{l}
\mu_{j+1}F_{j+1}(\lambda)=-\lambda m_{j+1}\nu_{j+1}G_{j+1}(\lambda)-b_{j+1}\mu_j F_j(\lambda)\\
\left(\lambda-\frac{k_{j+1}}{b_{j+1}}\right)\nu_{j+1}G_{j+1}(\lambda)=\mu_j F_j(\lambda)+(k_{j+1}-\lambda b_{j+1})\nu_{j} G_{j}(\lambda)
\end{array},\quad j=1,3
\right.
\end{eqnarray}
and
\begin{eqnarray}\label{ufb1}
\left\{
\begin{array}{l}
\mu_{j+1}F_{j+1}(\lambda)=-\lambda m_{j+1}\nu_{j+1}G_{j+1}(\lambda)+(k_{j+1}-\lambda b_{j+1})\mu_j F_j(\lambda)\\
\nu_{j+1}G_{j+1}(\lambda)=\mu_j F_j(\lambda)+(k_{j+1}-\lambda b_{j+1})\nu_{j} G_{j}(\lambda)
\end{array},\quad j=2,4.
\right.
\end{eqnarray}
With the above properties, we can also present the relationship between the roots of $F_j(\lambda)$ and $G_{j+1}(\lambda)$ for $j=2,4$. In  fact,  when $j=2,4$,  $s_{j+1}=s_j+1$ and by   \dref{hs}, \dref{ufb3} and \dref{kb>lam3},
$\mu_j F_j(\lambda)\ll -\nu_{j} (\lambda b_{j+1}-k_{j+1})G_{j}(\lambda)$.  Then,  \dref{hs},
 Lemma \ref{one}, \dref{ufb3} and  \dref{ufb1}  imply
  \begin{eqnarray}\label{ajbj+1l3}
\a_j(1)<\b_{j+1}(1)<\a_j(2)< \cdots<\a_j(s_j)< \b_{j+1}(s_{j+1}-1)< \a_{j+1}(s_{j+1})= \lambda_3,\quad j=2,4.
\end{eqnarray}

Now,  we prove \dref{exp1m} by using reduction to absurdity. Suppose
\begin{eqnarray}\label{mmsigma}
\max_{j\in[2,4]}\frac{m_j}{m_{j+1}}\leq\sigma\triangleq\frac{\lambda_1}{8\lambda_3}\left(1-\left(\frac{\lambda_2}{\lambda_3}\right)^{\frac{1}{3}}\right).
\end{eqnarray}
%
First, it is evident that for each  $j\in[1,4]$, \dref{ufb3}, \dref{ufb} and \dref{ufb1} yield
\begin{eqnarray}\label{ma}
m_{j+1}=-\frac{\mu_{j+1}}{\nu_{j+1}}\frac{F_{j+1}(\alpha_{j}(1))}{\alpha_{j}(1) G_{j+1}(\alpha_{j}(1))}&=&-\frac{\mu_{j+1}}{\nu_{j+1}}\frac{\prod_{i=1}^{s_{j+1}}(\alpha_{j}(1)-\alpha_{j+1}(i))}{\alpha_{j}(1)\prod_{i=1}^{s_{j+1}-1}(\alpha_{j}(1)-\beta_{j+1}(i))}\nonumber\\
&\geq &-\frac{\mu_{j+1}}{\nu_{j+1}}\frac{\alpha_{j}(1)-\alpha_{j+1}(1)}{\alpha_{j}(1)}.
\end{eqnarray}
In particular, when $j=1,3$, \dref{hs} further implies
\begin{eqnarray}\label{135j}
m_{j+1}=-\frac{\mu_{j+1}}{\nu_{j+1}}\frac{F_{j+1}(\alpha_{j}(s_j))}{\alpha_{j}(s_j) G_{j+1}(\alpha_{j}(s_{j}))}=-\frac{\mu_{j+1}}{\nu_{j+1}}
\frac{\prod_{i=1}^{s_{j+1}}(\lambda_3-\alpha_{j+1}(i))}{\lambda_3\prod_{i=1}^{s_{j+1}-1}(\lambda_3-\beta_{j+1}(i))}
<-\frac{\mu_{j+1}}{\nu_{j+1}}.
\end{eqnarray}
Note that by comparing the  leading coefficients of the polynomials  in \dref{ufb}--\dref{ufb1}, we assert that  for each $j\in[1,4]$ with $b_{j+1}>0$,
\begin{eqnarray}\label{fac}
-\frac{\mu_{j}}{\nu_{j}}=-\frac{\mu_{j+1}+m_{j+1}\nu_{j+1}}{\nu_{j+1}-\mu_{j}}.
\end{eqnarray}

Finally, let us complete the proof by considering the following four cases.
\\\\
\textbf{Case 1:} $b_3b_5>0$. First, we  estimate $m_j/m_{j+1}$   for $j=2,4$. Let $\varsigma_{j}=k_{j+1}/ b_{j+1} $, then $\varsigma_{j}>\alpha_{j+1}(s_{j+1})=\lambda_3$ by  \dref{ufb3}--\dref{kb>lam3}. Therefore,  noting that $\nu_{j+1}\mu_{j}>0$, \dref{hs} and \dref{fac} lead to
\begin{eqnarray}\label{fc0}
-\frac{\mu_{j}}{\nu_{j}}
&=&-\frac{\mu_{j+1}+m_{j+1}\nu_{j+1}}{\nu_{j+1}-\mu_{j}}>-\frac{\mu_{j+1}+m_{j+1}\nu_{j+1}}{\nu_{j+1}}=-\frac{\mu_{j+1}}{\nu_{j+1}}\left(1-\frac{F_{j+1}(\varsigma_{j})}{\varsigma_{j} G_{j+1}(\varsigma_{j})}\right)\nonumber\\
&=&-\frac{\mu_{j+1}}{\nu_{j+1}}\frac{\varsigma_{j}\prod_{i=1}^{s_{j+1}-1}(\varsigma_{j}-\beta_{j+1}(i))-\prod_{i=1}^{s_{j+1}}(\varsigma_{j}-\alpha_{j+1}(i))}{\varsigma_{j}\prod_{i=1}^{s_{j+1}-1}(\varsigma_{j}-\beta_{j+1}(i))}\nonumber\\
&>& -\frac{\mu_{j+1}}{\nu_{j+1}}\frac{\alpha_{j+1}(1)
\prod_{i=1}^{s_{j+1}-1}(\varsigma_{j}-\beta_{j+1}(i))}{\varsigma_{j}\prod_{i=1}^{s_{j+1}-1}(\varsigma_{j}-\beta_{j+1}(i))}\nonumber\\
&=&
-\frac{\mu_{j+1}}{\nu_{j+1}}\frac{\alpha_{j+1}(1)}{\varsigma_{j}}\geq
-\frac{\mu_{j+1}}{\nu_{j+1}}\frac{\lambda_1}{\varsigma_{j}},\quad
j=2,4.
\end{eqnarray}
So, if  $\varsigma_j\leq 2\lambda_3$, the above inequality reduces to
\begin{eqnarray} \label{fc1}
-\frac{\mu_{j}}{\nu_{j}}>
 -\frac{\mu_{j+1}}{\nu_{j+1}}\frac{\lambda_1}{2\lambda_3},\quad
j=2,4.
\end{eqnarray}
We next treat
 $\varsigma_2>2\lambda_3$. 
  Since \dref{hs} and \dref{ufb3} imply $\b_3(1)<\a_3(2)$ and
 $\lambda_1=\a_5(1)<\a_3(1)$,    then
 \begin{eqnarray*}
 (\varsigma_{2}-\lambda_1) G_{3}(\varsigma_{2})= (\varsigma_{2}-\lambda_1) (\varsigma_{2}-\b_3(1))> (\varsigma_{2}-\a_3(1)) (\varsigma_{2}-\a_3(2)) =   F_{3}(\varsigma_{2}),
\end{eqnarray*}
 which is equivalent to
$1-\frac{F_{3}(\varsigma_{2})}{\varsigma_{2} G_{3}(\varsigma_{2})}>\frac{\lambda_1}{\varsigma_2}.$
 Further, by \dref{ufb1}  again, $\mu_{2}/\nu_{3}=G_3(\varsigma_2)/F_2(\varsigma_2)$, so \dref{ajbj+1l3} shows
\begin{eqnarray}
-\frac{\mu_{2}}{\nu_{2}}&=&-\frac{\mu_{3}+m_{3}\nu_{3}}{\nu_{3}-\mu_{2}}=-\frac{\mu_3}{\nu_3} \frac{1-\frac{F_{3}(\varsigma_{2})}{\varsigma_{2} G_{3}(\varsigma_{2})}}{1-\frac{G_3(\varsigma_{2})}{F_2(\varsigma_2)}}>-\frac{\mu_3}{\nu_3} \frac{\lambda_1}{\varsigma_2}\frac{F_2(\varsigma_2)}{F_2(\varsigma_2)-G_3(\varsigma_2)}\label{mu2nu2>GF}\\
&= &-\frac{\mu_3}{\nu_3} \frac{\lambda_1}{\beta_3(1)-\alpha_2(1)} \frac{\varsigma_2-\alpha_2(1)}{\varsigma_2}>-\frac{\mu_3}{\nu_3}\frac{\lambda_1}{\lambda_3} \left(1-\frac{\lambda_3}{2\lambda_3}\right)>-\frac{\mu_3}{\nu_3}\frac{\lambda_1}{8\lambda_3}.\label{fc3}
\end{eqnarray}
As for  $\varsigma_4>2\lambda_3$,    similar to \dref{mu2nu2>GF},   we compute by \dref{ajbj+1l3} that
\begin{eqnarray}\label{fc2}
-\frac{\mu_{4}}{\nu_{4}}
&>&-\frac{\mu_5}{\nu_5} \frac{\lambda_1}{\varsigma_4}\frac{F_4(\varsigma_4)}{F_4(\varsigma_4)-G_5(\varsigma_4)}\nonumber\\
&=&-\frac{\mu_5}{\nu_5}\frac{\lambda_1}{\varsigma_4}\frac{(\varsigma_4-\alpha_4(1))(\varsigma_4-\alpha_4(2))}{(\varsigma_4-\alpha_4(1))(\varsigma_4-\alpha_4(2))-(\varsigma_4-\beta_5(1))(\varsigma_4-\beta_5(2))}\nonumber\\
&=&-\frac{\mu_5}{\nu_5} \frac{\lambda_1}{\varsigma_4}\frac{(\varsigma_4-\alpha_4(1))(\varsigma_4-\alpha_4(2))}{(\beta_5(1)+\beta_5(2)-\alpha_4(1)-\alpha_4(2))\varsigma_4+\alpha_4(1)\alpha_4(2)-\beta_5(1)\beta_5(2)}\nonumber\\
&>& -\frac{\mu_5}{\nu_5} \frac{\lambda_1}{\beta_5(1)+\beta_5(2)-\alpha_4(1)-\alpha_4(2)}
\frac{\varsigma_4-\alpha_4(1)}
{\varsigma_4}
\frac{\varsigma_4-\alpha_4(2)}
{\varsigma_4}\nonumber\\
&>& -\frac{\mu_5}{\nu_5} \frac{\lambda_1}{2\lambda_3}\left(1-\frac{\lambda_3}{2\lambda_3}\right)^2=-\frac{\mu_5}{\nu_5} \frac{\lambda_1}{8\lambda_3}.
\end{eqnarray}
As a result, by \dref{fc1},   \dref{fc3} and \dref{fc2},  the following inequality always holds:
\begin{eqnarray}
-\frac{\mu_{j}}{\nu_{j}}>-\frac{\mu_{j+1}}{\nu_{j+1}}\frac{\lambda_1}{8\lambda_3},\quad j=2,4, \nonumber
\end{eqnarray}
and thus \dref{ma} and \dref{135j} yield
\begin{eqnarray}\label{mjk1}
m_{j}>-\frac{\mu_{j+1}}{\nu_{j+1}}\frac{\lambda_1}{8\lambda_3}\frac{\alpha_{j-1}(1)-\alpha_{j}(1)}{\alpha_{j-1}(1)}>m_{j+1}\frac{\lambda_1}{8\lambda_3}\frac{\alpha_{j-1}(1)-\alpha_{j}(1)}{\alpha_{j-1}(1)},\quad j=2,4.
\end{eqnarray}
We proceed to the  calculation  of  $m_3/m_4$. Note that
 $\lambda_1<\b_4(1)<\a_4(2)<\alpha_{3}(2)=\lambda_3$,  
then  analogous  to \dref{fc0}, we can prove $-\frac{\mu_{3}}{\nu_{3}}>-\frac{\mu_{4}}{\nu_{4}}\frac{\lambda_1}{\lambda_3},$
which together with \dref{ma} and \dref{135j} leads to
\begin{eqnarray}\label{3m}
m_{3}\geq -\frac{\mu_{4}}{\nu_{4}}\frac{\lambda_1}{\lambda_3}\frac{\alpha_{2}(1)-\alpha_{3}(1)}{\alpha_{2}(1)}>m_{4}\frac{\lambda_1}{\lambda_3}\frac{\alpha_{2}(1)-\alpha_{3}(1)}{\alpha_{2}(1)}.
\end{eqnarray}

Now, combining \dref{mjk1} and \dref{3m}, \dref{mmsigma}    derives
\begin{eqnarray}
\frac{\lambda_1}{8\lambda_3}\frac{\alpha_{j}(1)-\alpha_{j+1}(1)}{\alpha_{j}(1)}\leq\sigma,\quad j=1,2,3. \nonumber
\end{eqnarray}
This is no other than
\begin{eqnarray}
\frac{\alpha_{j}(1)}{\alpha_{j+1}(1)}\leq \frac{\lambda_1}{\lambda_1-8\lambda_3 \sigma},\quad j=1,2,3.\nonumber
\end{eqnarray}
Consequently,
\begin{eqnarray}
\lambda_3=\alpha_1(1)\leq \left(\frac{\lambda_1}{\lambda_1-8\lambda_3 \sigma}\right)^3 \alpha_4(1)<\left(\frac{\lambda_1}{\lambda_1-8\lambda_3 \sigma}\right)^3 \lambda_2,\nonumber
\end{eqnarray}
which contradicts to
 the definition of $\sigma$ in \dref{mmsigma}.\\\\
\textbf{Case 2:} $b_3=b_5=0$. Then,  \dref{ufb1} infers
\begin{eqnarray}
m_3=-\frac{\mu_3}{\nu_3}\quad\mbox{and}\quad m_5=-\frac{\mu_5}{\nu_5}.\nonumber
\end{eqnarray}
In this case, $\nu_5=\mu_4$ and then \dref{ufb1} becomes
\begin{eqnarray}\label{ufb1b50}
\left\{
\begin{array}{l}
\mu_{5}(F_{5}(\lambda)-\lambda G_{5}(\lambda))=k_{5}\nu_5 F_4(\lambda)\\
\mu_{4}(G_{5}(\lambda)- F_4(\lambda))=k_{5}\nu_{4} G_{4}(\lambda)
\end{array}.
\right.
\end{eqnarray}
Since $\deg(F_{5}(\lambda)-\lambda G_{5}(\lambda))=2$ and $\deg(G_{5}(\lambda)- F_4(\lambda))=1$, comparing the leading
coefficients of the polynomials  in \dref{ufb1b50}, we calculate
\begin{eqnarray*}
-\frac{\mu_{4}}{\nu_{4}}\left(\sum_{i=1}^{2}\beta_{5}(i)-\sum_{i=1}^{2}\alpha_{4}(i)\right)
=-\frac{\mu_{5}}{\nu_{5}}\left(\sum_{i=1}^{3}\alpha_{5}(i)-\sum_{i=1}^{2}\beta_{5}(i)\right)=k_5.
\end{eqnarray*}
Consequently, by \dref{hs} and \dref{ajbj+1l3},
\begin{eqnarray}\label{45}
-\frac{\mu_{4}}{\nu_{4}}=-\frac{\mu_{5}}{\nu_{5}}\frac{\sum_{i=1}^{3}\alpha_{5}(i)-\sum_{i=1}^{2}\beta_{5}(i)}{\sum_{i=1}^{2}\beta_{5}(i)-\sum_{i=1}^{2}\alpha_{4}(i)}>-\frac{\mu_{5}}{\nu_{5}}\frac{\alpha_{5}(1)}{2\lambda_3}>\frac{\lambda_1}{2\lambda_3}m_5,
\end{eqnarray}
and hence
\begin{eqnarray}
m_3=-\frac{\mu_3}{\nu_3}=-\frac{\mu_{4}+m_{4}\nu_{4}}{\nu_{4}-\mu_{3}}>-\frac{\mu_{4}}{\nu_{4}}-m_4
=\frac{\lambda_1}{2\lambda_3}m_5-m_4.\nonumber
\end{eqnarray}
Therefore, by  \dref{mmsigma},
$$
\sigma^2 m_5\geq m_3>\frac{\lambda_1}{2\lambda_3}m_5-m_4\geq \frac{\lambda_1}{2\lambda_3}m_5-\sigma m_5,
$$
which infers $\sigma^2+\sigma>\frac{\lambda_1}{2\lambda_3}$. This is impossible due to the definition of  $\sigma$ in \dref{mmsigma}.\\\\
\textbf{Case 3:} $b_3=0$ and $b_5>0$. Then, $m_3=-\frac{\mu_3}{\nu_3}$ and a similar argument as \dref{45} indicates
\begin{eqnarray*}
-\frac{\mu_{2}}{\nu_{2}}=-\frac{\mu_{3}}{\nu_{3}}\frac{\sum_{i=1}^{2}\alpha_{3}(i)-\beta_{3}(1)}{\beta_{3}(1)-\alpha_{2}(1)},
\end{eqnarray*}
and then by \dref{hs} and \dref{ajbj+1l3},
\begin{eqnarray*}
m_2&=&-\frac{\mu_{2}}{\nu_{2}}\frac{\lambda_3-\alpha_2(1)}{\lambda_3}=-\frac{\mu_{3}}{\nu_{3}}\frac{\lambda_3-\alpha_2(1)}{\lambda_3}\frac{\sum_{i=1}^{2}\alpha_{3}(i)-\beta_{3}(1)}{\beta_{3}(1)-\alpha_{2}(1)}\\
&>&-\frac{\mu_{3}}{\nu_{3}}\frac{\sum_{i=1}^{2}\alpha_{3}(i)-\beta_{3}(1)}{\lambda_3}>m_3\frac{\lambda_1}{\lambda_3}.\nonumber
\end{eqnarray*}
It contradicts to  \dref{mmsigma} that  $\frac{m_2}{m_3}\leq\sigma<\frac{\lambda_1}{\lambda_3}$.\\\\
\textbf{Case 4:} $b_3>0$ and $b_5=0$. Then, \dref{45} holds and according to \dref{ma},
\begin{eqnarray}
m_{4}>-\frac{\mu_{4}}{\nu_{4}}\frac{\alpha_{3}(1)-\alpha_{4}(1)}{\alpha_{3}(1)}>m_{5}\frac{\lambda_1}{2\lambda_3}\frac{\alpha_{3}(1)-\alpha_{4}(1)}{\alpha_{3}(1)}.\nonumber
\end{eqnarray}
By \dref{fc0}--\dref{fc3} and \dref{3m}
 in \textbf{Case 1}, we can demonstrate
\begin{eqnarray}
m_{j}>m_{j+1}\frac{\lambda_1}{8\lambda_3}\frac{\alpha_{j-1}(1)-\alpha_{j}(1)}{\alpha_{j-1}(1)},\quad j=2,3.\nonumber
\end{eqnarray}
The rest of the proof thus keeps the same as that in \textbf{Case 1}.
\qed
\section{Proof of Remark \ref{constants}}\label{appc}
We only prove \dref{ib}, as the other results are trivial.
First, note that $C_{2}(m-1)>1$ and 
 $\Lambda/\Delta\geq 2$, so
\begin{eqnarray}
(1+C_2(m-1))^n \rho_{m-1}&<& 2^n C_2^n(m-1)\rho_{m-1}=2^n\left(\frac{2^{2n(n+1)^2}\Lambda^{n(n+1)}}{\Delta^{n^2}\varepsilon^{n}}\right)\varepsilon^{n+1}\nonumber\\
&<&\frac{2^{2(n+1)^3}\Lambda^{(n+1)^2}}{2\Delta^{n^2+n}}\varepsilon<\frac{\Delta}{2n}<\frac{\Delta}{2}.\nonumber
\end{eqnarray}
Let $m>2$. Since $2\varepsilon<1$,  for each $j\in[1,m-2]$,
\begin{eqnarray}
(1+C_2(j))^n \rho_j&<&2^n C_2^n(j)\rho_{j}=2^n\left(\frac{2^{2n(n+1)^2}\Lambda^{n(n+1)}}{\Delta^{n^2}\varepsilon^{n(n+1)^{m-j-1}}}\right)\varepsilon^{(n+1)^{m-j}}\nonumber\\
&=&2^n\left(\frac{2^{2n(n+1)^2}\Lambda^{n(n+1)}}{\Delta^{n^2}}\right)\left(\varepsilon^{(n+1)^{m-j-2}}\right)^{n+1}\nonumber\\
&<&\frac{2^{2n(n+1)^2}\Lambda^{n(n+1)}}{\Delta^{n^2}}\varepsilon^{(n+1)^{m-j-2}}=C_2^n(j+1)\rho_{j+1}<(1+C_2(j+1))^n \rho_{j+1}.\nonumber
\end{eqnarray}
When $m>3$, for each $j\in[2,m-2]$,
\begin{eqnarray}
\frac{n(1+C_{2}(j-1))^n \rho_{j-1}}{\Delta}&<&  \frac{n2^n C_{2}^n(j-1)\rho_{j-1}}{\Delta}=n2^n \left(\frac{2^{2n(n+1)^2}\Lambda^{n(n+1)}}{\Delta^{n^2+1}}\right)\varepsilon^{(n+1)^{m-j}}\nonumber\\
&=& n2^n \left(\frac{2^{2n(n+1)^2}\Lambda^{n(n+1)}}{\Delta^{n^2+1}}\right)\varepsilon^{(n+1)^{m-j}-(n+1)^{m-j-1}}\rho_{j+1}\nonumber\\
&<& n2^n \left(\frac{2^{2n(n+1)^2}\Lambda^{n(n+1)}\varepsilon}{\Delta^{n^2+1}}\right)\rho_{j+1}\nonumber\\
&<&\frac{\Delta^{n}}{2^{(n+1)^3}\Lambda^{n+1}}\rho_{j+1}\leq  \frac{1}{4}\left(\frac{\Delta^n}{2^{(n+1)^2}\Lambda^{n+1}}\right)\rho_{j+1}.\nonumber
\end{eqnarray}
Hence, \dref{ib} follows immediately.
\section{Proofs of Lemmas \ref{cas1}--\ref{cas2}}\label{appb}
\begin{proof}[Proof of Lemma \ref{cas1}]
Let $p\geq 1$.  We first take a number $\lambda^{*}$ satisfying
\begin{eqnarray}\label{lam*}
\lambda^{*}>\a_p\quad \mbox{and}\quad -\frac{\mu}{\nu}\frac{F(\lambda^{*})}{\lambda^{*}G(\lambda^{*})}=m^{*}.
\end{eqnarray}
Such $\lambda^{*}$ indeed exists because of  \dref{FG}--\dref{zcs}, which yield   $-\frac{\mu}{\nu}\frac{F(\a_p)}{\a_pG(\a_p)}=0 $ and
\begin{eqnarray}
\lim_{\lambda\rightarrow+\infty}-\frac{\mu}{\nu}\frac{F(\lambda)}{\lambda G(\lambda)}=-\frac{\mu}{\nu}>CM>m^{*}.\nonumber
\end{eqnarray}
Then, let
\begin{eqnarray}\label{d1}
\left\{
\begin{array}{l}
F_{0}(\lambda)=\frac{\mu F(\lambda)+ m^{*}\nu \lambda G(\lambda)}{(\mu+m^{*}\nu)(\lambda-\lambda^{*})}\\
b^{*}=-\frac{\mu+m^{*}\nu}{\nu} \frac{F_{0}(\lambda^{*})}{G(\lambda^{*})}\\
\mu_{0}=\nu\frac{G(\lambda^{*})}{F_{0}(\lambda^{*})}\\
G_{0}(\lambda)=\frac{\nu G(\lambda)-\mu_{0}F_{0}(\lambda)}{(\nu-\mu_{0})(\lambda-\lambda^{*})}\\
\nu_{0}=\frac{\mu_{0}-\nu}{b^{*}}
\end{array},
\right.
\end{eqnarray}
we shall show that all the above defined numbers and polynomials fulfill our requirements.

Observe that $\mu\nu<0$ and by \dref{aplam}--\dref{zcs},
$$\alpha_{p}\geq\lambda_p>\lambda_{p-1}+(1+C_2(p-1))^{n}\rho_{p-1}>\beta_{p-1}>\a_{p-1}>\cdots>\b_1>\a_1>0,$$
 then $\lambda G(\lambda)\ll F(\lambda)$. As a result, by \dref{lam*},
\begin{eqnarray}\label{anb2}
-\frac{\mu+m^{*}\nu}{\nu}=-\frac{\mu}{\nu}\left(1-\frac{F(\lambda^{*})}{\lambda^{*}G(\lambda^{*})}\right)>  -\frac{\mu}{\nu}\left(1-\frac{\prod_{i=1}^{p}(\lambda^{*}-\alpha_{i})}{(\lambda^{*}-\alpha_{1})\prod_{i=1}^{p-1}(\lambda^{*}-\alpha_{i+1})}\right)=0.
\end{eqnarray}
Since \dref{lam*} indicates  $(\lambda-\lambda^{*})|(\mu F(\lambda)+\lambda m^{*}\nu G(\lambda))$, \dref{anb2}
means $F_{0}(\lambda)$ is a well-defined monic polynomial of degree $p-1$. 
 Recall that $\lambda G(\lambda)\ll F(\lambda)$ and $|\mu|>m^*|\nu|$,
 applying Lemma \ref{lemma1} to polynomials  $\mu F(\lambda)$ and $-m^{*}\nu\lambda G(\lambda)$ shows
\begin{eqnarray}\label{alp}
\alpha_{i}'\in (\alpha_{i}, \beta_{i}),\quad i=1,\ldots,p-1.
\end{eqnarray}
 Therefore,
 \begin{eqnarray}\label{F0FG}
\lambda F_{0}(\lambda)\ll F(\lambda)\quad\mbox{and}\quad
F_{0}(\lambda)\ll  G(\lambda),
\end{eqnarray}
which give $
b^{*}=-\frac{\mu+m^{*}\nu}{\nu}\left( \frac{F_{0}(\lambda^{*})}{G(\lambda^{*})}\right)>0,
$ 
\begin{eqnarray}\label{mu0nu}
\mu_{0}\nu=\nu^2\left(\frac{G(\lambda^{*})}{F_{0}(\lambda^{*})}\right)>0\quad\mbox{and}\quad
|\mu_{0}|=|\nu|\left(\frac{G(\lambda^{*})}{F_{0}(\lambda^{*})}\right)<|\nu|.
\end{eqnarray}
Consequently, by \dref{d1} and \dref{anb2},
\begin{eqnarray} -\frac{\mu_{0}}{\nu_{0}}=-\frac{\mu+m^{*}\nu}{\nu-\mu_{0}}=-\frac{\mu+m^{*}\nu}{\nu}\frac{1}{1-\mu_0/\nu}   >-\frac{\mu}{\nu}-m^{*}.
\end{eqnarray}
So far, we have verified that $\lambda^*,b^*, \mu_{0}$ and $\nu_{0}$ satisfy the condition of Lemma \ref{cas1}.
For these numbers,  the first equality of \dref{cas1e} follows directly from \dref{d1}. Considering \dref{alp}, if the second inequality of \dref{sxb} holds when $p>2$, then $F_0(\lambda)$ in \dref{d1} will  be the  exact polynomial desired   for both statements (i) and (ii).

Next, we check $G_0(\lambda)$ in \dref{d1}.
The definition of $\mu_{0}$ infers
$(\lambda-\lambda^{*})\big|(\nu G(\lambda)-\mu_{0}F_{0}(\lambda))$,
hence  \dref{mu0nu} implies that   $G_{0}(\lambda)$ is a well-defined monic polynomial. We discuss this part by considering two cases.\\
(i) $p=2$. In this case, $G_0(\lambda)=1$ fulfills the second equality of \dref{cas1e} by \dref{d1}.\\
(ii) $p>2$. Taking account to \dref{F0FG} and \dref{mu0nu}, we apply
Lemma \ref{lemma1} to polynomials  $\nu G(\lambda)$ and $\mu_{0}F_{0}(\lambda)$, then
\begin{eqnarray}\label{bet}
\beta_{i}'\in (\beta_{i}, \alpha_{i+1}'),\quad  i=1,\ldots,p-2,
\end{eqnarray}
so the roots of $G_0(\lambda)$ are distinct.
Further, we can verify the second equality of \dref{cas1e}  from \dref{d1} again.

Now, it remains to show the second inequality of \dref{sxb} for $i=1,\ldots,p-2$
when $p>2$. Combining \dref{alp} and \dref{bet}, it gives $\lambda G_{0}(\lambda)\ll F_{0}(\lambda)$.  Fix an index $i\in[1,p-2]$, we compute
\begin{eqnarray}\label{bi-aj}
\prod_{j=1}^{p-1}(\beta_{i}'-\alpha_{j}')=F_{0}(\beta_{i}')=\frac{1}{\mu_{0}}(\nu G(\beta_{i}')-b^{*}(\lambda^{*}-\beta_{i}')\nu_{0} G_{0}(\beta_{i}'))=\frac{F_{0}(\lambda^{*})}{G(\lambda^{*})}\prod_{j=1}^{p-1}(\beta_{i}'-\beta_{j}).
\end{eqnarray}
Note that  $\frac{F_{0}(\lambda^{*})}{G(\lambda^{*})}>1$ by \dref{F0FG} and $
\prod_{j>i+1}\frac{\beta_{i}'-\beta_{j}}{\beta_{i}'-\alpha_{j}'}\geq 1
$ by \dref{alp} and \dref{bet}, \dref{bi-aj} immediately
leads to
\begin{eqnarray}\label{b-a/b-b}
\frac{\beta_{i}'-\alpha_{i}'}{\beta_{i}'-\beta_{i}}>\frac{\beta_{i}'-\beta_{i+1}}{\beta_{i}'-\alpha_{i+1}'}\prod_{j<i}\frac{\beta_{i}'-\beta_{j}}{\beta_{i}'-\alpha_{j}'}.
\end{eqnarray}

We are going to  employ Lemma \ref{lemma1} to estimate term $\frac{\beta_{i}'-\beta_{i+1}}{\beta_{i}'-\alpha_{i+1}'}$ in \dref{b-a/b-b}. For this,
denote $\eta=\frac{1}{2}\min_{j\in[1,p-1]}(\beta_{j}-\alpha_{j})$. Recall that $C_1,\rho_1<1$,   then by \dref{zcs}, for each $j\in[1,p-1]$,
$$
C_1^{n}\rho_1\leq\min\{1, C_1^{n}\rho_j\}<\min\{1,\b_j-\a_j\},
$$
which means $C_1^{n}\rho_1/2<\min\{1,\eta\}$. As a result, by
 \dref{MC},
\begin{eqnarray}
\left|\frac{m^{*}\nu}{\mu}\right|&<&\frac{m^{*}}{CM}<\left(\frac{C_1^{n}\rho_1}{2}\right)^{2n}\frac{1}{2+2\Lambda^n}<\frac{\min\lbrace 1, \eta^{2p}\rbrace}{2+2\Lambda^{p}}\nonumber\\
&<&\frac{\min\left\lbrace 1,\left(\frac{1}{2}\min_{1\leqslant l\leqslant p-1}\left(\alpha_{l+1}-\alpha_{l}\right)\right)^{p}\right\rbrace\min\lbrace 1, \eta^{p}\rbrace}{2+2\Lambda^{p}}\nonumber\\
&<&\frac{\min\left\lbrace 1,\left(\min_{1\leqslant l\leqslant p-1}(\alpha_{l+1}-\alpha_{l})-\eta\right)^{p}\rbrace\min\lbrace 1, \eta^{p}\right\rbrace}{2+2\max_{j\in[1,p]}\left|\alpha_{j}\prod_{l=1}^{p-1}(\alpha_{j}-\beta_{l})\right|}.\nonumber
\end{eqnarray}
So, by applying Lemma \ref{lemma1} to  polynomials $\mu F(\lambda)$ and $-m^{*}\nu\lambda G(\lambda)$, we conclude
\begin{eqnarray}\label{er}
\max_{j\in[1,p-1]}(\alpha_{j}'-\alpha_{j})<\eta=\frac{1}{2}\min_{j\in[1,p-1]}(\beta_{j}-\alpha_{j}).
\end{eqnarray}
Then, $\beta_{i+1}-\alpha_{i+1}'\geq \frac{\beta_{i+1}-\alpha_{i+1}}{2}$, which together with \dref{zcs} indicates
\begin{eqnarray}\label{fbe}
\frac{\beta_{i}'-\beta_{i+1}}{\beta_{i}'-\alpha_{i+1}'}&=&1+\frac{\beta_{i+1}-\alpha_{i+1}'}{\alpha_{i+1}'-\beta_{i}'}\geq  1+\frac{\beta_{i+1}-\alpha_{i+1}}{2(\beta_{i+1}-\beta_{i}')}\nonumber\\
&\geq &1+\frac{C_1^{n}\rho_{i+1}}{2(\lambda_{i+1}-\beta_{i}')}\geq  1+\frac{\Delta^n}{2^{(n+1)^2}\Lambda^{n+1}}\rho_{i+1}.
\end{eqnarray}
Next,  we deal with
$
\prod_{j<i}\frac{\beta_{i}'-\beta_{j}}{\beta_{i}'-\alpha_{j}'}
$ in \dref{b-a/b-b} for $i\geq 2$.  As a matter of fact,
   by 
  \dref{zcs}, for any $j<i$,
\begin{eqnarray}\label{ben}
\frac{\beta_{i}'-\beta_{j}}{\beta_{i}'-\alpha_{j}'}&>&\frac{\beta_{i}'-\beta_{j}}{\beta_{i}'-\alpha_{j}}=1-\frac{\beta_{j}-\alpha_{j}}{\beta_{i}'-\alpha_{j}}\nonumber\\
&>& 1-\frac{\beta_{j}-\lambda_j}{\beta_{i}'-\lambda_{j}} >1-\frac{\beta_{j}-\lambda_j}{\lambda_i-\lambda_{j}}\geq 1-\frac{(1+C_2(j))^n \rho_{j}}{\Delta}.
\end{eqnarray}

Now, if $i\geq 2$,   substituting \dref{fbe} and \dref{ben} into \dref{b-a/b-b} yields
\begin{eqnarray}\label{babb}
\frac{\beta_{i}'-\alpha_{i}'}{\beta_{i}'-\beta_{i}}
&>&\left(1+\frac{\Delta^n}{2^{(n+1)^2}\Lambda^{n+1}}\rho_{i+1}\right)\prod_{j<i}\left(1-\frac{(1+C_2(j))^n \rho_{j}}{\Delta}\right)\nonumber\\
&\geq& \left(1+\frac{\Delta^n}{2^{(n+1)^2}\Lambda^{n+1}}\rho_{i+1}\right) \left(1-\frac{(1+C_2(i-1))^n \rho_{i-1}}{\Delta}\right)^{i-1}\nonumber\\
&>&\left(1+\frac{\Delta^n}{2^{(n+1)^2}\Lambda^{n+1}}\rho_{i+1}\right)\left(1-\frac{(1+C_2(i-1))^n \rho_{i-1}}{\Delta}\right)^n,
\end{eqnarray}
where the second inequality follows from \dref{ib} in Remark \ref{constants}. Since the
\textit{Bernoulli inequality} gives
$$
\left(1-\frac{(1+C_2(i-1))^n \rho_{i-1}}{\Delta}\right)^n>1-\frac{n(1+C_2(i-1))^n \rho_{i-1}}{\Delta},
$$
\dref{babb} reduces to

\begin{eqnarray}\label{ii>}
\frac{\beta_{i}'-\alpha_{i}'}{\beta_{i}'-\beta_{i}}&>&\left(1+\frac{\Delta^n}{2^{(n+1)^2}\Lambda^{n+1}}\rho_{i+1}\right) \left(1-\frac{n(1+C_2(i-1))^n \rho_{i-1}}{\Delta}\right)\nonumber\\
&>&1+\frac{\Delta^n}{2^{(n+1)^2+1}\Lambda^{n+1}}\rho_{i+1}.
\end{eqnarray}
As for  $i=1$, it is trivial that
$
\frac{\beta_{i}'-\alpha_{i}'}{\beta_{i}'-\beta_{i}}> 1+\frac{\Delta^n}{2^{(n+1)^2}\Lambda^{n+1}}\rho_{i+1}.
$
So, both the two cases lead to \dref{ii>}.
Hence,
\begin{eqnarray}\label{dy}
\frac{\beta_{i}'-\alpha_{i}'}{\beta_{i}-\alpha_{i}}<\frac{\beta_{i}'-\alpha_{i}'}{\beta_{i}-\alpha_{i}'}=1+\frac{1}{\frac{\beta_{i}'-\alpha_{i}'}{\beta_{i}'-\beta_{i}}-1}\leq 1+\frac{2^{(n+1)^2+1}\Lambda^{n+1}}{\Delta^n \rho_{i+1}}<C_2(i).
\end{eqnarray}
On the other hand, in view of \dref{er},
\begin{eqnarray}\label{>C1}
\frac{\beta_{i}'-\alpha_{i}'}{\beta_{i}-\alpha_{i}}>\frac{\beta_{i}-\alpha_{i}'}{\beta_{i}-\alpha_{i}}\geq \frac{1}{2}\geq C_1.
\end{eqnarray}
Therefore,  \dref{sxb} is a direct result of  \dref{dy} and \dref{>C1}.
\end{proof}

\begin{proof}[Proof of Lemma \ref{cas2}]
(i) Let $p\geq 2$ and set
\begin{eqnarray}\label{d2}
\left\{
\begin{array}{l}
m^{*}=-\frac{\mu}{\nu}\frac{F(\lambda^{*})}{\lambda^{*}G(\lambda^{*})}\\
F_{0}(\lambda)=\frac{\mu F(\lambda)+ m^{*}\nu\lambda G(\lambda)}{\mu+m^{*}\nu}
\end{array}.
\right.
\end{eqnarray}
By \dref{FG} and \dref{zcs}, it is clear that $m^{*}>0$.
Observe that
\begin{eqnarray}
\left|\frac{m^{*}\nu}{\mu}\right|=\frac{\prod_{j=1}^{p}(\lambda^{*}-\alpha_{j})}{\lambda^{*}\prod_{j=1}^{p-1}(\lambda^{*}-\beta_{j})}<\frac{\prod_{j=1}^{p}(\lambda^{*}-\alpha_{j})}{(\lambda^{*}-\alpha_{1})\prod_{j=1}^{p-1}(\lambda^{*}-\alpha_{j+1})}=1,\nonumber
\end{eqnarray}
then $F_{0}(\lambda)$ is a well-defined monic polynomial of
$
\deg(F_{0}(\lambda))=\deg(F(\lambda))=p.
$
Now,  \dref{zcs} indicates  $\lambda G(\lambda)\ll F(\lambda)$,
 by applying Lemma \ref{lemma1} to polynomials $\mu F(\lambda)$ and  $- m^{*}\nu \lambda G(\lambda)$, it follows that the first $p-1$ roots of $F_{0}(\lambda)$ satisfy
\begin{eqnarray}\label{alal}
\alpha_{j}'\in (\alpha_{j}, \beta_{j}),\quad j=1,\ldots, p-1.
\end{eqnarray}
Moreover,  the definition of $m^*$ in \dref{d2} shows
     $$(\lambda-\lambda^{*})|(\mu F(\lambda)+ m^{*}\nu \lambda G(\lambda)),$$
which yields
$
\alpha_{p}'=\lambda^{*}.
$
Hence,  $F(\lambda)\ll F_{0}(\lambda)$.

Next, let
\begin{eqnarray}\label{d3}
\left\{
\begin{array}{l}
\mu_{0}=\tau \nu\\
b^{*}=-\frac{\mu+m^{*}\nu}{\mu_{0}}\\
\nu_{0}=\frac{\mu_{0}-\nu}{b^{*}}\\
G_{0}(\lambda)=\frac{\nu(\lambda-\lambda^{*})G(\lambda)-\mu_{0}F_{0}(\lambda)}{b^{*}\nu_{0}(\lambda^{*}-\lambda)}
\end{array},
\right.
\end{eqnarray}
where
\begin{eqnarray}
\tau=\left\{
\begin{array}{l}
v_1,\quad p>2\\
v_2,\quad p=2\\
\end{array},
\right.\nonumber
\end{eqnarray}
with
\begin{eqnarray}
v_1&=&\frac{1}{2}\frac{\min\lbrace 1,\left(\min_{1\leqslant l\leqslant p-2}(\beta_{l+1}-\beta_{l})-\eta_1\right)^{p-1}\rbrace\min\lbrace 1, \eta_1^{p-1}\rbrace}{2+2\max_{j\in[1,p-1]}\left|\prod_{h=1}^{p-1}(\beta_{j}-\alpha_{h}')\right|},\nonumber\\
v_2&=&\frac{1}{4} \min\left\lbrace \frac{\eta_2}{\beta_{1}-\alpha_{1}'},1\right\rbrace,\nonumber\\
\eta_1 &=&\min\left\lbrace \frac{\min_{1\leqslant l\leqslant p-2}(\beta_{l+1}-\beta_{l})}{4}, \min_{1\leqslant l\leqslant p-1}(\beta_{l}-\alpha_{l})\right\rbrace,\nonumber\\
\eta_2 &=&\min\left\lbrace\beta_{1}-\alpha_{1},\frac{1}{2}\right\rbrace.\nonumber
\end{eqnarray}
Therefore,
\begin{eqnarray}
b^{*}=-\frac{\mu+m^{*}\nu}{\mu_{0}}=-\frac{\mu}{\tau \nu}\left(1-\frac{F(\lambda^{*})}{\lambda^{*}G(\lambda^{*})}\right)>0.\nonumber
\end{eqnarray}
Note that $\tau\in (0,1)$, then \dref{d3} gives $\mu_0/\nu\in (0,1)$. As a result, by $\lambda G(\lambda)\ll F(\lambda)$,
\begin{eqnarray}
-\frac{\mu_{0}}{\nu_{0}}=-\frac{\mu+m^{*}\nu}{\nu-\mu_{0}}>-\frac{\mu}{\nu}\left(1-\frac{F(\lambda^{*})}{\lambda^{*}G(\lambda^{*})}\right)
>-\frac{\mu}{\nu}\frac{\lambda^{*}G(\lambda^{*})-(\lambda^{*}-\alpha_{1})G(\lambda^{*})}{\lambda^{*}G(\lambda^{*})}>-\frac{\lambda_1}{\Lambda}\frac{\mu}{\nu}>0.\nonumber
\end{eqnarray}
So, $G_{0}(\lambda)$ is a  well-defined monic polynomial of degree $p-1$.  Since  \dref{alal} means that $\frac{F_{0}(\lambda)}{(\lambda-\lambda^{*})}\ll G(\lambda)$, by applying Lemma \ref{lemma1} to polynomials $\nu G(\lambda)$ and $\mu_{0}\frac{F_{0}(\lambda)}{(\lambda-\lambda^{*})}$,  we deduce $\beta_{p-1}'>\beta_{p-1}$ and
\begin{eqnarray}\label{b'l3}
\beta_j' \in (\beta_j, \alpha_{j+1}'),\quad j=1,\ldots, p-2.
\end{eqnarray}
Observe that  plugging \dref{d2} and \dref{d3} into \dref{cas2e} immediately shows the validity of \dref{cas2e}.

At last,  we show the second inequality of \dref{sxb} for $i=1,\ldots,p-1$. Fix an index $i\in[1,p-1]$. By plugging $\lambda=\beta_{i}$ into \dref{cas2e}, we obtain
$
\mu F(\beta_{i})=-b^{*}\mu_{0}F_{0}(\beta_{i}),
$
which equals to
\begin{eqnarray}\label{prodba}
\prod_{j=1}^{p}(\beta_{i}-\alpha_{j}')=\frac{1}{1-\frac{F(\lambda^{*})}{\lambda^{*}G(\lambda^{*})}}\prod_{j=1}^{p}(\beta_{i}-\alpha_{j}).
\end{eqnarray}
Note that
$
\frac{\beta_{i}-\alpha_{j}}{\beta_{i}-\alpha_{j}'}>1
$
 for all $j<i$ because of \dref{alal}. As for  $j\in[i+1,p-1]$,
by   \dref{zcs}  and \dref{alal},
\begin{eqnarray}
\frac{\beta_{i}-\alpha_{j}}{\beta_{i}-\alpha_{j}'}=1-\frac{\alpha_{j}'-\alpha_{j}}{\alpha_{j}'-\beta_{i}}\geq 1-\frac{\beta_{j}-\lambda_j}{\alpha_{j}-\beta_{i}}\geq 1-\frac{\beta_{j}-\lambda_j}{\lambda_{i+1}-\beta_{i}},\nonumber
\end{eqnarray}
which together with   \dref{ib} yields
\begin{eqnarray}
\frac{\beta_{i}-\alpha_{j}}{\beta_{i}-\alpha_{j}'}\geq  1-\frac{\beta_{j}-\lambda_j}{(\lambda_{i+1}-\lambda_i)-(\beta_{i}-\lambda_i)}  \geq 1-\frac{(1+C_{2}(j))^{n}\rho_j}{2\Delta-(1+C_{2}(i))^n \rho_i}>\frac{1}{2}.\nonumber
\end{eqnarray}
Furthermore, since $\alpha_{p}\in[\lambda_p, \lambda^{*})\subset [\lambda_p, \Lambda)$, \dref{vDelta}, \dref{ib} and \dref{zcs}  lead to
\begin{eqnarray}
\frac{\beta_{i}-\alpha_{p}}{\beta_{i}-\alpha_{p}'}\geq \frac{\lambda_{i+1}-\beta_{i}}{\Lambda}\geq \frac{\Delta-(1+C_2^n(i))\rho_i}{\Lambda}>\frac{\Delta}{2\Lambda}.\nonumber
\end{eqnarray}
Consequently, it follows from \dref{b'l3} and \dref{prodba} that
\begin{eqnarray}
\frac{\beta_{i}'-\alpha_{i}'}{\beta_{i}-\alpha_{i}}>\frac{\beta_{i}-\alpha_{i}'}{\beta_{i}-\alpha_{i}}=\frac{1}{1-\frac{F(\lambda^{*})}{\lambda^{*}G(\lambda^{*})}}\prod_{j\not= i}\frac{\beta_{i}-\alpha_{j}}{\beta_{i}-\alpha_{j}'}>\frac{1}{2^n}\frac{\Delta}{2\Lambda}=C_1.\nonumber
\end{eqnarray}

Now, we prove $\beta_{i}'-\alpha_{i}'\leq C_2(i)(\beta_{i}-\alpha_{i})$ for each $i\in [1,p-1]$. If $p>2$,
\begin{eqnarray}
\left|\frac{\mu_{0}}{\nu}\right|=\tau=v_1<\frac{\min\lbrace 1,\left(\min_{1\leqslant l\leqslant p-2}(\beta_{l+1}-\beta_{l})-\eta_1\right)^{p-1}\rbrace\min\lbrace 1, \eta_1^{p-1}\rbrace}{2+2\max_{j\in[1,p-1]}\left|\prod_{h=1}^{p-1}(\beta_{j}-\alpha_{h}')\right|}.\nonumber
\end{eqnarray}
By applying Lemma \ref{lemma1} to polynomials $\nu G(\lambda)$ and $\mu_{0}\frac{F_{0}(\lambda)}{(\lambda-\lambda^{*})}$, it infers
\begin{eqnarray}
\max_{i\in[1,p-1]}(\beta_{i}'-\beta_{i})<\eta_1\leq \min_{1\leqslant i\leqslant p-1}(\beta_{i}-\alpha_{i}).\nonumber
\end{eqnarray}
When $p=2$,
\begin{eqnarray}
\left|\frac{\mu_{0}}{\nu}\right|=\tau<\frac{1}{2} \min\left\lbrace \frac{\eta_2}{\beta_{1}-\alpha_{1}'},1\right\rbrace\nonumber
\end{eqnarray}
and  applying Lemma \ref{lemma1} to polynomials $\nu G(\lambda)$ and $\mu_{0}\frac{F_{0}(\lambda)}{(\lambda-\lambda^{*})}$ shows
\begin{eqnarray}
\beta_{1}'-\beta_{1}<\eta_2\leq \beta_{1}-\alpha_{1}.\nonumber
\end{eqnarray}
So both cases result in
\begin{eqnarray}
\max_{i\in[1,p-1]}(\beta_{i}'-\beta_{i})<\min_{1\leqslant i\leqslant p-1}(\beta_{i}-\alpha_{i}),\nonumber
\end{eqnarray}
which implies that $\beta_{i}'-\alpha_{i}<2(\beta_{i}-\alpha_{i})$ for each $i\in [1,p-1]$.
Then,
\begin{eqnarray}
\beta_{i}'-\alpha_{i}'<\beta_{i}'-\alpha_{i}<2(\beta_{i}-\alpha_{i})\leq C_2(i)(\beta_{i}-\alpha_{i}).\nonumber
\end{eqnarray}
This finishes the proof of statement (i).\\\\
(ii) Let $m^{*}=-\frac{\mu}{\nu}\frac{\lambda^{*}-\alpha_1}{\lambda^{*}}$, $\mu_{0}=\frac{1}{2} \nu$,
 $b^{*}=-\frac{\mu}{\nu}\frac{2\alpha_1}{\lambda^{*}}$ and $\nu_{0}=\frac{1}{4}\frac{\nu^2}{\mu}\frac{\lambda^{*}}{\alpha_1}$,
  we can directly compute that $ F_0(\lambda)=\lambda-\lambda^{*}$ and  $ G_{0}(\lambda)=1$ satisfy equation \dref{cas2e}.
\end{proof}



\end{document}